\documentclass{amsart}

\usepackage{mathtools,amsthm,amssymb,mathrsfs,enumitem,mdframed}
\usepackage{scalerel}
\usepackage{mathscinet}

\numberwithin{equation}{section}
\numberwithin{figure}{section}
\theoremstyle{plain}
\newtheorem{thm}{Theorem}[section]
  \theoremstyle{definition}
  \newtheorem{defn}[thm]{Definition}
  \newtheorem*{defn*}{Definition}
  \theoremstyle{remark}
  \newtheorem*{rem*}{Remark}
  \theoremstyle{plain}
  \newtheorem{lem}[thm]{Lemma}
  \theoremstyle{plain}
  \newtheorem{prop}[thm]{Proposition}
  \newtheorem{cor}[thm]{Corollary}
  \theoremstyle{definition}
  \newtheorem{example}[thm]{Example}

\makeatletter
\newcommand{\norm}{\@ifstar{\@normb}{\@normi}}
\newcommand{\@normb}[2]{\left\Vert{#1}\right\Vert_{#2}}
\newcommand{\@normi}[2]{\Vert{#1}\Vert_{#2}}
\makeatother

\global\long\def\Leb#1{L^{#1}}
\DeclareMathOperator{\supp}{supp}

\global\long\def\bes#1#2{B_{#1}^{#2}}
\global\long\def\Bessel#1#2{L_{#1}^{#2}}
\global\long\def\oBessel#1#2{{L}^{#2}_{#1,0}}
\global\long\def\haSob#1#2{{\mathcal{H}^{#2}_{#1}}}
\newcommand{\action}[1]{\left<#1 \right>}
\newcommand{\boldb}{\mathbf{b}}

\DeclareMathOperator{\Tr}{Tr}
\DeclareMathOperator{\Div}{div}

\newcommand{\relphantom}[1]{\mathrel{\phantom{#1}}}
\newcommand{\myd}[1]{\,d#1}

\newcommand{\Di}{\scaleobj{0.75}{D}}
\newcommand{\Neu}{\scaleobj{0.75}{N}}

\usepackage[commentmarkup=todo]{changes}
\definechangesauthor[name=Kim, color=blue]{kim}
\definechangesauthor[name=Kwon, color=red]{kwon}
\usepackage{subcaption}
\usepackage[numbers,sort&compress]{natbib}
\begin{document}
 
\title{}
 
\title[Elliptic equations with singular drifts on Lipschitz domains]{Dirichlet and Neumann problems for elliptic equations with singular drifts on Lipschitz domains}
\author{Hyunseok Kim}
\address{(04107) Department of Mathematics, Sogang Univeristy, 35, Baekbeom-ro, Mapo-gu, Seoul, Republic of Korea}
\email{kimh@sogang.ac.kr}
\thanks{The researh of this work was supported by Basic Science Research Program through the National Research Foundation of Korea (NRF) funded by the Ministry of Education (No. NRF-2016R1D1A1B02015245).}

\author{Hyunwoo Kwon}
\address{(04107) Department of Mathematics, Sogang Univeristy, 35, Baekbeom-ro, Mapo-gu, Seoul, Republic of Korea}
\curraddr{(28187) Department of Mathematics, Republic of Korea Air Force Academy, Postbox 335-2, 635, Danjae-ro   Sangdang-gu, Cheongju-si  Chung\-cheong\-buk-do, Republic of Korea}
\email{willkwon@afa.ac.kr; willkwon@sogang.ac.kr}


\begin{abstract}
We consider the Dirichlet and Neumann problems for second-order linear elliptic equations:
\[
-\triangle u +\Div(u\boldb) =f \quad\text{ and }\quad  -\triangle v -\boldb \cdot \nabla v =g
\]
in a bounded Lipschitz domain $\Omega$ in $\mathbb{R}^n$ $(n\geq 3)$, where $\boldb:\Omega \rightarrow \mathbb{R}^n$ is a given vector field. Under the assumption that $\boldb \in L^{n}(\Omega)^n$, we first establish existence and uniqueness of solutions in $L_{\alpha}^{p}(\Omega)$ for the Dirichlet   and Neumann problems. Here $L_{\alpha}^{p}(\Omega)$ denotes the Sobolev space (or Bessel potential space) with the pair $(\alpha,p)$ satisfying certain conditions. These results extend the classical works of Jerison-Kenig \cite{MR1331981} and Fabes-Mendez-Mitrea \cite{MR1658089} for the Poisson equation.  We also  prove existence and uniqueness of solutions of the Dirichlet problem with boundary data in $L^{2}(\partial\Omega)$. Our results for the Dirichlet problems hold  even for the case $n=2$.  
\end{abstract}

\maketitle
 
\bibliographystyle{amsplain} 

\section{Introduction}
In this paper we study the Dirichlet and Neumann problems  for second-order linear elliptic equations  with singular drifts on a  bounded Lipschitz domain $\Omega$ in $\mathbb{R}^n$ $(n\geq 2)$.  Given a vector field $\boldb=\left(b^{1},\dots,b^{n}\right):\Omega\rightarrow\mathbb{R}^{n}$, we consider the following Dirichlet problems:
\begin{equation}\label{eq:D-1}\tag{$D$}
\left\{\begin{alignedat}{2}
-\triangle u+\Div\left(u\boldb\right) & =f & \quad & \text{in }\Omega,\\
u & =u_{\Di} & \quad & \text{on }\partial\Omega \nonumber
\end{alignedat}\right.
\end{equation}
and
\begin{equation}\label{eq:D-2}\tag{$D'$}
\left\{\begin{alignedat}{2}
-\triangle v-\boldb\cdot\nabla v & =g & \quad & \text{in }\Omega,\\
v & =v_{\Di} & \quad & \text{on }\partial\Omega.
\end{alignedat}\right.
\end{equation}
 We also consider the following Neumann
problems:
\begin{equation}\label{eq:N-1}\tag{$N$}
\left\{\begin{alignedat}{2}
-\triangle u+\Div\left(u\boldb\right) & =f & \quad & \text{in }\Omega,\\
\left(\nabla  u - u\boldb \right)\cdot \nu & =u_{\Neu} & \quad & \text{on }\partial\Omega
\end{alignedat}\right.
\end{equation}
and
\begin{equation}\label{eq:N-2}\tag{$N'$}
\left\{\begin{alignedat}{2}
-\triangle v-\boldb\cdot\nabla v & =g & \quad & \text{in }\Omega,\\
\nabla v\cdot \nu & =v_{\Neu} & \quad & \text{on }\partial\Omega,\nonumber
\end{alignedat}\right.
\end{equation}
where  $\nu$ denotes the outward unit normal to the boundary  $\partial\Omega$.  For $-\infty<\alpha<\infty$, $0<\beta<1$, and $1<p<\infty$, we denote $\Bessel{\alpha}{p}(\Omega)$ and $\bes{\beta}{p}(\partial\Omega)$ the Sobolev space (or Bessel potential space) and Besov space, respectively.  See Section \ref{sec:prelim} for more details on these function spaces.

When $\boldb$ is sufficiently regular, for example, $\boldb \in \Leb{\infty}(\Omega)^n$, unique solvability results  in $\Bessel{1}{p}(\Omega)$  are well-known for the Dirichlet  and  Neumann problems on smooth domains.    For a singular drift  $\boldb \in \Leb{n}(\Omega)^n$ with $n\geq 3$,  existence and uniqueness results in $\Bessel{1}{2}(\Omega)$ for the Dirichlet problems  have been already obtained by Trudinger \cite{T73} and by  Droniou \cite{MR1908676} on general Lipschitz domains. The corresponding results for the Neumann problems were established by Droniou-V\'azquez \cite{MR2476418}. Recently, $\Bessel{1}{p}$-results were obtained by Kim-Kim \cite{MR3328143} for the Dirichlet problems on $C^1$-domains and by Kang-Kim \cite{MR3623550}  for the Dirichlet and Neumann problems on domains which have  small Lipschitz constant.  The authors in  \cite{T73,MR1908676,MR2476418,MR3328143,MR3623550} also obtained $\Bessel{1}{p}$-results on two-dimensional domains $\Omega$ for the case when $\boldb \in \Leb{r}(\Omega)^2$ for some $r>2$. For a singular drift  $\boldb$ in the critical space $\Leb{2}(\Omega)^2$, the second author \cite{K21} recently proved unique solvability results in $\Bessel{1}{p}(\Omega)$ for the Dirichlet problems on bounded domains in $\mathbb{R}^2$ which have small Lipschitz constant, by using the recent results of Krylov \cite{MR4223034,MR4317707}.  Moreover, several authors have studied regularity properties of  solutions of the Dirichlet problems (see \cite{MR3048265,MR3818670,MR2667641,MR2760150} and references therein).  Kim-Ryu-Woo \cite{KRW20} recently obtained solvability results in Sobolev spaces with mixed norms for parabolic equations with unbounded drifts.

The assumption $\boldb \in \Leb{n}(\Omega)^n$ is essential to our study due to simple examples.  Let $\Omega$ be the  unit ball in $\mathbb{R}^n$ centered at the origin. Define $\boldb(x) = -n{x}/{|x|^2}$ and $v(x)= 1-|x|^2$. Then $\boldb \in \Leb{q}(\Omega)^n$ for all $q<n$ and  $v\in \Bessel{1}{2}(\Omega)$  is a nontrivial solution of \eqref{eq:D-2} with $g=0$ and $v_{\Di}=0$. This example shows that   solutions of the problem \eqref{eq:D-2} may not be unique  when $\boldb \in \Leb{q}(\Omega)^n$ for any $q<n$.  

For more than 40 years, many authors have studied the Dirichlet and Neumann problems for the Poisson equation on Lipschitz domains.  In particular, Jerison-Kenig \cite{MR1331981} established an optimal solvability result in $\Bessel{\alpha}{p}(\Omega)$ for the Dirichlet problem for the Poisson equation when $(\alpha,p)$ belongs to a set $\mathscr{A}$ (see Definition \ref{defn:admissible} for a precise definition of $\mathscr{A}$).  Similar results for the Neumann problem were obtained by  Fabes-Mendez-Mitrea \cite{MR1658089} for $n\geq 3$ and Mitrea \cite{MR1883390} for $n=2$.

The purpose of this paper is twofold. First, extending  the classical results in \cite{MR1658089,MR1331981}  and the recent results in  \cite{MR1908676,MR2476418,MR3623550,MR3328143}, we prove existence and uniqueness of solutions in $L_{\alpha}^{p}(\Omega)$ for the Dirichlet  problems on bounded Lipschitz domains $\Omega$ in $\mathbb{R}^n$, $n\geq 2$. Similar results  will be also obtained for the Neumann problems on Lipschitz domains in $\mathbb{R}^n$ for $n\geq 3$.  The second purpose  is  to establish  the unique solvability in a function space $X$ for the Dirichlet problem  \eqref{eq:D-1}  with   boundary data $u_D$ in $L^2(\partial\Omega)$.    Such a problem with singular data is often motivated by  finite element analysis for some optimal control problems (see \cite{MR2084239,MR3070527} and references therein).   A relevant result to our purpose is due to  Choe-Kim \cite{MR2846167} who proved a unique solvability result for the stationary incompressible Navier-Stokes system with $L^2$-boundary data on Lipschitz domains $\Omega$ in $\mathbb{R}^3$. Motivated by this work, we
take $X$ as the sum   of Sobolev spaces   $L_{\alpha}^{p}(\Omega)$ and the space of harmonic functions in $L_{1/2}^{2}(\Omega)$. Then  existence, uniqueness, and regularity of solutions  in $X$ will be proved for the Dirichlet problem \eqref{eq:D-1} with boundary data in $\Leb{2}(\partial\Omega)$ on general bounded   Lipschitz domains $\Omega$ in $\mathbb{R}^n$. 

Our main results are stated precisely in Section \ref{sec:main-results} after basic notions and preliminary results are introduced in Section \ref{sec:prelim}. For the Dirichlet problems, Theorem \ref{thm:Dirichlet-problems} shows  existence and uniqueness of solutions in  $\Bessel{\alpha}{p}(\Omega)$ of the problem \eqref{eq:D-1}  for all pairs $(\alpha,p)$ belonging to  a subset $\mathscr{A}\cap\mathscr{B}$ of $(0,2)\times (0, \infty )$. Here $\mathscr{B}$ is a set of   pairs $(\alpha,p)$ such that  $\Div(u\boldb) \in \Bessel{\alpha-2}{p}(\Omega)$ for all $u\in \Bessel{\alpha}{p}(\Omega)$ (see Definition \ref{defn:admissible-B}). The same theorem also shows  unique solvability   in Sobolev spaces for the dual  problem  \eqref{eq:D-2}. The case of $\Leb{2}$-boundary data  is then considered in Theorem \ref{thm:L2-boundary}, which   shows   existence and uniqueness of a solution $u$ of the problem \eqref{eq:D-1} for every   $f\in \Bessel{\alpha-2}{p}(\Omega)$ and $u_{\Di} \in \Leb{2}(\partial\Omega)$, where $(\alpha,p)$ belongs to   $\mathscr{A}\cap\mathscr{B}$.  The  solution $u$ is given by $u=u_1 +u_2$ for some $u_1 \in \Bessel{1/2}{2}(\Omega)$ tending to $u_{\Di}$ nontangentially a.e.\ on $\partial\Omega$ and $u_2 \in \Bessel{\alpha}{p}(\Omega)+\Bessel{1}{\frac{2n}{n+1}}(\Omega)$ having zero trace.  Moreover, we deduce a regularity property of the solution $u$: that is,  if $u_{\Di} \in \Leb{q}(\partial\Omega)$ and $2\leq  q\leq \infty$, then $u\in \Bessel{1/q}{q}(\Omega)+\Bessel{\alpha}{p}(\Omega)$.

To state our results for the Neumann problems,  let $\action{\cdot,\cdot}$ denote the dual pairing between a Banach space $X$ and its dual space $X'$. It will be shown in Theorem \ref{thm:Neumann} that if $(\alpha,p)\in \mathscr{A}\cap\mathscr{B}$, then for each $f\in (\Bessel{2-\alpha}{p'}(\Omega))'$ and $u_{\Neu} \in (\bes{1+1/p-\alpha}{p'}(\partial\Omega))'$ satisfying the compatibility condition $\action{f,1}+\action{u_{\Neu},1}=0$, there exists a unique function $u\in \Bessel{\alpha}{p}(\Omega)$ with  $\int_\Omega u \myd{x}=0$ such that
\[  \action{\nabla u,\nabla \phi} -\action{u\boldb,\nabla \phi} = \action{f,\phi} + \action{u_{\Neu},\Tr \phi} \quad \text{for all } \phi \in \Bessel{2-\alpha}{p'}(\Omega).
 \]
Here    $\Tr$ denotes the trace operator given in Theorem \ref{thm:trace-theorem}.  A similar $\Bessel{\alpha}{p}$-result will   also be   proved for the dual problem \eqref{eq:N-2}.  However, an explicit counterexample (Example \ref{example:Neumann-fails}) suggests that Theorem \ref{thm:Neumann} may not imply   solvability for the Neumann problems \eqref{eq:N-1} and \eqref{eq:N-2}. Introducing a generalized normal trace operator $\gamma_{\nu}$ (Proposition \ref{prop:normal-trace}), we will show in Theorem \ref{thm:real-Neumann-problem} that if the   data $f$ is sufficiently regular, then there exists a unique $u\in \Bessel{\alpha}{p}(\Omega)$ with  $\int_\Omega u \myd{x}=0$ such that
\begin{equation*}
\left\{\begin{alignedat}{2}
-\triangle u+\Div\left(u\boldb\right) & =f & \quad & \text{in }\Omega,\\
\gamma_{\nu}(\nabla  u - u\boldb) & =u_{\Neu} & \quad & \text{on }\partial\Omega,
\end{alignedat}\right.
\end{equation*}
which provides a solvability result for  the Neumann problem  \eqref{eq:N-1}.
We also have a similar result for the dual problem \eqref{eq:N-2}.

Theorems \ref{thm:Dirichlet-problems} and \ref{thm:Neumann} are proved by a functional analytic argument.  To estimate the drift terms, we derive bilinear estimates  which are inspired by Gerhardt \cite{MR520820}; see Lemmas \ref{lem:basic-estimates} and \ref{lem:Gerhardt} below.  To prove Theorem \ref{thm:Dirichlet-problems}, we  reduce the problems \eqref{eq:D-1} and \eqref{eq:D-2} to the problems with trivial boundary data by using a trace theorem (Theorem \ref{thm:trace-theorem}) and the bilinear estimates. For a fixed $(\alpha,p)\in \mathscr{A}\cap\mathscr{B}$, let  $\mathcal{L}_{\alpha,p}^{D}:\oBessel{\alpha}{p}(\Omega)\rightarrow \Bessel{\alpha-2}{p}(\Omega)$ be the operator associated with the Dirichlet problem \eqref{eq:D-1}, that is,
\[  \mathcal{L}_{\alpha,p}^{D}u=-\triangle u +\Div(u\boldb).\]
Here $\oBessel{\alpha}{p}(\Omega)$ is the space of all functions in $\Bessel{\alpha}{p}(\Omega)$ whose trace is zero  (see Theorem \ref{thm:trace-theorem}).  Similarly, we denote by  $\mathcal{L}^{*,D}_{2-\alpha,p'}:\oBessel{2-\alpha}{p'}(\Omega)\rightarrow \Bessel{-\alpha}{p'}(\Omega)$  the operator associated with the dual problem \eqref{eq:D-2}. Due to the bilinear estimates, these operators are bounded linear operators. Also, the estimates (see Lemmas \ref{lem:Gerhardt} and  \ref{lem:compactness-of-operators}) enable us to use  the Riesz-Schauder theory to conclude that the operator $\mathcal{L}_{\alpha,p}^{D}$ (or $\mathcal{L}_{2-\alpha,p'}^{*,D}$)  is bijective if and only if it is injective.  Uniqueness results in $\Bessel{1}{p}(\Omega)$ for the problem  \eqref{eq:D-2} have been shown by {Trudinger \cite{T73}} and Droniou \cite{MR1908676} for $p=2<n$ and by the second author \cite{K21} for $p>n=2$. Using these results, we prove that the kernels of the operator $\mathcal{L}_{1,p}^{D}$ and $\mathcal{L}_{1,p'}^{*,D}$  are trivial when $(1,p)\in \mathscr{A}\cap\mathscr{B}$ (Lemma \ref{lem:L-1p-result}). For general $(\alpha,p)\in \mathscr{A}\cap\mathscr{B}$, we  use a regularity  lemma (Lemma \ref{lem:Sobolev-regularity}) to prove that the kernel of $\mathcal{L}_{\alpha,p}^{D}$ is trivial (Proposition \ref{prop:uniqueness-of-D}). This implies the unique solvability in $\Bessel{\alpha}{p}(\Omega)$ for the Dirichlet problem \eqref{eq:D-1}.  By duality, we  obtain a similar result for the dual problem \eqref{eq:D-2}.  This outlines the proof of Theorem \ref{thm:Dirichlet-problems}.

A similar strategy works for the proof of Theorem \ref{thm:Neumann}. For a fixed $(\alpha,p)\in \mathscr{A}\cap\mathscr{B}$, let $\mathcal{L}_{\alpha,p}^{N}$ and $\mathcal{L}_{2-\alpha,p'}^{*,N}$ be the operators associated with the Neumann problems \eqref{eq:N-1} and \eqref{eq:N-2}, respectively.  The characterization of the kernels $\mathcal{L}_{1,2}^{N}+\lambda \mathcal{I}_{2}$ and $\mathcal{L}_{1,2}^{*,N}+\lambda \mathcal{I}_{2}$ was already obtained by Droniou-V\'azquez \cite{MR2476418} and Kang-Kim \cite{MR3623550} for all $\lambda \geq 0$, where the operator $\mathcal{I}_{p}:\Leb{p}(\Omega)\rightarrow (\Leb{p'}(\Omega))'$ is defined by $\action{\mathcal{I}_{p}u,v} = \int_\Omega uv\myd{x}$.   Following a similar scheme as in  the Dirichlet problems, we will show that  the kernels of the operators $\mathcal{L}_{\alpha,p}^{N}+\lambda \mathcal{I}_{p}$ and $\mathcal{L}_{2-\alpha,p'}^{*,N}+\lambda \mathcal{I}_{p'}$ are trivial for all $\lambda>0$. Also, we will show that the kernels of  $\mathcal{L}_{\alpha,p}^{N}$ and $\mathcal{L}_{2-\alpha,p'}^{*,N}$ are one dimensional (see Proposition \ref{prop:Neumann-kernel-characterization}). Then  Theorem \ref{thm:Neumann} follows from the Riesz-Schauder theory again.  This outlines the proof of Theorem \ref{thm:Neumann}.

The existence and regularity results in Theorem \ref{thm:L2-boundary}  easily follow from  Theorems \ref{thm:L2-nontangential} and \ref{thm:Dirichlet-problems}. For  the uniqueness part, we shall prove    an embedding result in $\Bessel{\alpha}{p}(\Omega)$ (Lemma \ref{lem:embedding-for-uniqueness}) and a lemma for the nontangential behavior of a solution (Lemma \ref{lem:L2-boundary-nontangential-convergence}). Finally, Theorem \ref{thm:real-Neumann-problem} will be deduced  from Proposition \ref{prop:normal-trace} and Theorem \ref{thm:Neumann}.

The rest of this paper is organized as follows. In Section \ref{sec:prelim}, we summarize known results for functions spaces on Lipschitz domains and unique solvability results for the Dirichlet and Neumann problems for the Poisson equation on Lipschitz domains. We also derive bilinear estimates which will be used repeatedly in this paper.    Section \ref{sec:main-results} is devoted to  presenting the main results of this paper for the Dirichlet problems with boundary data in $\bes{\alpha}{p}(\partial\Omega)$ and in $\Leb{2}(\partial\Omega)$, respectively.  We also state the main results for the Neumann problems. The proofs of all the main results   are provided in Sections \ref{sec:Dirichlet} and \ref{sec:Neumann}.

\subsection*{Acknowledgement} The authors would like to express sincere gratitude to  anonymous referees for valuable comments and suggestions, which have greatly improved the manuscript.  

\section{Preliminaries}\label{sec:prelim}
Throughtout the paper, we assume that $\Omega$ is a bounded Lipschitz domain in $\mathbb{R}^n$, $n\geq 2$. By $C=C(p_1,\dots,p_k)$, we denote a generic positive constant depending only on the parameters $p_1,\dots,p_k$. For two Banach spaces $X$ and $Y$ with $X\subset Y$, we say that $X$ is continuously embedded into $Y$ and write  $X\hookrightarrow Y$  if there is a constant $C>0$ such that $\norm{x}{Y}\leq C \norm{x}{X}$ for all $x\in X$. For a Banach space $X$, we denote by $X$  the dual space of $X$. The dual pairing between $X$ and $X'$ is  denoted by $ \action{\cdot , \cdot }$.

\subsection{Embedding and trace results}

For $-\infty<\alpha<\infty$ and $1< p<\infty$, let
\[  \Bessel{\alpha}{p}(\mathbb{R}^n) = \{ (I-\triangle)^{-\alpha/2} f : f\in \Leb{p}(\mathbb{R}^n) \} \]
denote the Sobolev space (or Bessel potential space)  on $\mathbb{R}^n$ (see \cite{MR0482275,MR0290095,MR3243741}).  We denote by $\oBessel{\alpha}{p}(\Omega)$ and  $\Bessel{\alpha}{p}(\Omega)$ the Sobolev spaces on $\Omega$ defined as follows:
\begin{align*}
\oBessel{\alpha}{p}(\Omega) & = \{  u \in \Bessel{\alpha}{p}(\mathbb{R}^n) : \supp u \subset \overline{\Omega} \}, \\
\Bessel{\alpha}p(\Omega) & =\begin{cases}
\left\{ u\,|_{\Omega}:u\in\Bessel{\alpha}p\left(\mathbb{R}^{n}\right)\right\}  & \text{if }\alpha\geq0,\\
\left[\oBessel{-\alpha}{p'}(\Omega)\right]' & \text{if } \alpha<0,
\end{cases}
\end{align*}
where $p'=p/(p-1)$ is the conjugate exponent  to $p$. The norms on $\oBessel{\alpha}{p}(\Omega)$ and $\Bessel{\alpha}{p}(\Omega)$ are defined by 
\[   \norm{u}{\oBessel{\alpha}{p}(\Omega)}= \norm{u}{\Bessel{\alpha}{p}(\mathbb{R}^n)} \]
and 
\[  \norm{g}{\Bessel{\alpha}{p}(\Omega)} = \inf \{ \norm{u}{\Bessel{\alpha}{p}(\mathbb{R}^n)} : u\in \Bessel{\alpha}{p}(\Omega), u|_\Omega =g \},\]
respectively.  It was shown in  \cite[Remark 2.7, Proposition 2.9]{MR1331981} that $C_0^\infty(\Omega)$ and $C^\infty(\overline{\Omega})$ are dense in $\oBessel{\alpha}{p}(\Omega)$ and $\Bessel{\alpha}{p}(\Omega)$, respectively.  It was also shown in \cite[Propositions 2.9 and 3.5]{MR1331981} that
\begin{equation}\label{eq:zero-equal}
\oBessel{\alpha}{p}(\Omega) = \Bessel{\alpha}{p}(\Omega)\quad \text{if } 0\leq \alpha<\frac{1}{p}
\end{equation}
and
\[  \oBessel{-\alpha}{p}(\Omega) = \left[\Bessel{\alpha}{p'}(\Omega) \right]'\quad \text{for all } \alpha\geq 0.  \]

For $0<\alpha<1$ and $1<p<\infty$, we define
\[
\bes{\alpha}p(\mathbb{R}^{n-1})=\left\{ g\in\Leb p(\mathbb{R}^{n-1}):\int_{\mathbb{R}^{n-1}}\int_{\mathbb{R}^{n-1}}\frac{\left|g\left(x\right)-g\left(y\right)\right|^{p}}{\left|x-y\right|^{n-1+\alpha p}}\myd{x}dy<\infty\right\}.
\]
The norm on $\bes{\alpha}{p}(\mathbb{R}^{n-1})$ is defined by 
\[  \norm{g}{\bes{\alpha}{p}(\mathbb{R}^{n-1})}=\norm{g}{\Leb{p}(\mathbb{R}^{n-1})}+\left(\int_{\mathbb{R}^{n-1}}\int_{\mathbb{R}^{n-1}} \frac{|g(x)-g(y)|^p}{|x-y|^{n-1+\alpha p}} \myd{x}dy \right)^{1/p}.  \] 
Let $(Z_i,\varphi_i)_{1\leq i\leq m}$ be a coordinate pair associated with $\partial\Omega$ (see e.g. Verchota \cite{MR769382}). Choose a partition of unity $\{\Psi_i\}_{i=1}^m$ so that $\Psi_i \in C_c^\infty(\mathbb{R}^n)$, $\supp \Psi_i \subset Z_i$, and $\sum_{i=1}^m \Psi_i=1$ in a neighborhood of $\partial\Omega$.  For $0<\alpha<1$ and $1<p<\infty$, we define  $\bes{\alpha}{p}(\partial\Omega)$ as the space of all locally integrable functions $f$ on $\partial\Omega$ such that \[ (\Psi_i f)(\cdot,\varphi_i(\cdot))\in \bes{\alpha}{p}(\mathbb{R}^{n-1})\quad \text{for every } i=1,2, ..., m,  \]
endowed with the norm
\[  \norm{f}{\bes{\alpha}{p}(\partial\Omega)}=\left(\sum_{i=1}^m \norm{(\Psi_i f)(\cdot,\varphi_i(\cdot))}{\bes{\alpha}{p}(\mathbb{R}^{n-1})}^p \right)^{1/p}.  \] 
{We also define}
\[    \bes{-\alpha}{p}(\partial\Omega) =\left[\bes{\alpha}{p'}(\partial\Omega) \right]'. \]
See Jerison-Kenig \cite{MR1331981} and Fabes-Mendez-Mitrea \cite{MR1658089} for a basic theory of Sobolev spaces $\Bessel{\alpha}{p}(\Omega)$ and Besov spaces $\bes{\alpha}{p}(\partial\Omega)$ on Lipschitz domains.

We recall the following  embedding results  for Sobolev and Besov spaces.

\begin{thm}
\label{thm:Sobolev-embedding} Let $0\leq \beta\leq \alpha <\infty$ and $1<p,q<\infty$.
\begin{enumerate}[label={\textnormal{(\roman*)}},ref={\roman*},leftmargin=2em]
\item If $(\alpha,p)$ and $(\beta,q)$ satisfy
\begin{equation}\label{eq:Sobolev-embedding-condition}
\frac{1}{p}-\frac{\alpha}{n}\leq \frac{1}{q}-\frac{\beta}{n},
\end{equation}
then $\Bessel{\alpha}{p}(\Omega) \hookrightarrow \Bessel{\beta}{q}(\Omega),$ that is, $\Bessel{\alpha}{p}(\Omega)$ is continuously embedded into $\Bessel{\beta}{q}(\Omega)$. 
\item  If $\beta<\alpha$ and inequality \eqref{eq:Sobolev-embedding-condition} is strict, then the embedding $\Bessel{\alpha}{p}(\Omega)\hookrightarrow \Bessel{\beta}{q}(\Omega)$  is compact.
\item If $(\alpha,p)$ and $(\beta,q)$ satisfy
\[  0<\alpha<1\quad \text{and}\quad \frac{1}{p}-\frac{\alpha}{n-1} \leq \frac{1}{q}-\frac{\beta}{n-1},    \]
then
\[ \bes{\alpha}{p}(\partial\Omega)\hookrightarrow \begin{cases}
\bes{\beta}{q}(\partial\Omega)  & \text{if } \beta>0,\\
\Leb{q}(\partial\Omega) & \text{if }  \beta =0.
\end{cases}
\]
\end{enumerate}
\end{thm}
\begin{proof}
The proofs of (i) and (ii) can be found in the standard references (see e.g. \cite[p.60  and Proposition 4.6]{MR2250142}).     To prove (iii), we recall  the following embedding result:
\[ \bes{\alpha}{p}(\mathbb{R}^{n-1})\hookrightarrow \begin{cases}
\bes{\beta}{q}(\mathbb{R}^{n-1})  & \text{if } \beta>0,\\
\Leb{q}(\mathbb{R}^{n-1}) & \text{if }  0=\beta <\alpha
\end{cases}
\]
whenever  $(\alpha,p)$ and $(\beta,q)$ satisfy
\[    0\leq \beta\leq \alpha<1,\quad 1<p,q<\infty,\quad \text{and }\quad  \frac{1}{p}-\frac{\alpha}{n-1} = \frac{1}{q}-\frac{\beta}{n-1}\]
 (see e.g \cite[p.60 and Theorem 1.73]{MR2250142} and \cite[Theorem 7.34]{MR2424078}).
Then (iii) follows by using a partition of unity for the boundary $\partial\Omega$. 
\end{proof} 

{The following  trace and extension theorems are due to Jonsson \cite[Theorems 1 and 2]{MR543499} (see also Jonsson-Wallin \cite[Theorems 1 and 3, Chapter VII]{MR820626}) and the characterization of $\oBessel{\alpha}{p}(\Omega)$ is proved by Jerison-Kenig \cite[Proposition 3.3]{MR1331981}.}

\begin{thm}
\label{thm:trace-theorem} Let  $(\alpha,p)$ satisfy
\[ 1<p<\infty\quad  \text{and }\quad \frac{1}{p}<\alpha<1+\frac{1}{p}.\]
\begin{enumerate}[label={\textnormal{(\roman*)}},ref={\roman*},leftmargin=2em]
\item There exists a unique bounded linear operator $\Tr : \Bessel{\alpha}{p}(\Omega)\rightarrow \bes{\alpha-1/p}{p}(\partial\Omega)$ such that
\[ \Tr u = u\,|_{\partial\Omega}\quad \text{for all } u\in C^\infty(\overline{\Omega}).  \] 
\item $\oBessel{\alpha}{p}(\Omega)$ is the space of all functions $u$ in $\Bessel{\alpha}{p}(\Omega)$ with $\Tr u=0$.
\item There exists a bounded linear operator $\mathcal{E}:\bes{\alpha-1/p}{p}(\partial\Omega)\rightarrow \Bessel{\alpha}{p}(\Omega)$ such that
\[ \Tr \circ \,\mathcal{E}=I. \] 
\end{enumerate}
\end{thm}

Immediately from Theorems \ref{thm:Sobolev-embedding} and \ref{thm:trace-theorem}, we obtain the following result which is necessary for our study on the Dirichlet problem \eqref{eq:D-1} with $\Leb{2}$-boundary data.

\begin{cor}
\label{cor:trace-control}Let $(\alpha,p)$ satisfy
\[
1<p<\infty,\quad \frac{1}{p}<\alpha<1+\frac{1}{p} ,\quad \text{and }\quad  \frac{1}{p}-\frac{\alpha}{n}\leq \frac{n-1}{2n}.
\]
Then for all $u\in \Bessel{\alpha}{p}(\Omega)$, we have
\[u\in \Leb{\frac{2n}{n-1}}(\Omega),\quad  \Tr u \in \Leb{2}(\partial\Omega),\]
and
\begin{equation}\label{eq:Tr-constant}
\norm u{\Leb{\frac{2n}{n-1}}(\Omega)}+\norm{\Tr u}{\Leb 2(\partial\Omega)}\le C\norm u{\Bessel{\alpha}p(\Omega)}
\end{equation}
for some constant $C=C(n,\alpha,p,\Omega)>0$. 
\end{cor} 
\begin{rem*}
 It is obvious from the defintion that the embedding constant in Theorem \ref{thm:Sobolev-embedding} (i) depends on $\Omega$ in terms of its volume only. Our proof of Theorem \ref{thm:Sobolev-embedding} (iii) shows that the embedding constant depends  on $\Omega$ through its Lipschitz character. Also, it was shown by Jonsson \cite[Theorems 1 and 2]{MR543499}  that the  norm of the trace operator in Theorem \ref{thm:trace-theorem} (i) depends only on $n$, $\alpha$, $p$, and the Lipschitz character of $\Omega$ (see also Jonsson-Wallin \cite[Theorems 1 and 3, Chapter VII]{MR820626}). Therefore, the constant $C$ in \eqref{eq:Tr-constant} depends on the Lipschitz domain $\Omega$ only through its volume and Lipschitz character.
\end{rem*}

For a smooth vector field $F\in C^{\infty}(\overline{\Omega})^n$, integration by parts gives
\begin{equation}\label{eq:integration-by-parts-formula}
\int_{\partial \Omega} (F\cdot \nu) \phi \myd{\sigma} =  \int_\Omega F \cdot \nabla \phi \myd{x} + \int_\Omega  (\Div F) \phi \myd{x} \quad \text{for all } \phi \in C^{\infty}(\overline{\Omega}).
\end{equation}
This formula can be generalized for $F\in \Bessel{\alpha}{p}(\Omega)^n$ having divergence in $\Bessel{\beta-1}{q}(\Omega)$ with $(\alpha,p)$ and $(\beta,q)$ satisfying certain conditions. To show this, we first note that for $1<p<\infty$ and $-1/p'<\alpha<1/p$,   the pairing between $\Bessel{\alpha}{p}(\Omega)$ and $\Bessel{-\alpha}{p'}(\Omega)$ is well-defined. Indeed,  since $\oBessel{\beta}{q}(\Omega) =\Bessel{\beta}{q}(\Omega)$ for $0\leq \beta <1/q$, we have
\begin{equation}\label{eq:pairing-1}
\Bessel{-\alpha}{p'}(\Omega) = \left[\oBessel{\alpha}{p}(\Omega)\right]' = \left[\Bessel{\alpha}{p}(\Omega)\right]' \quad \text{if } \alpha \geq 0
\end{equation}
and
\begin{equation}\label{eq:pairing-2}
\Bessel{\alpha}{p}(\Omega) = \left[\oBessel{-\alpha}{p'}(\Omega)\right]' = \left[\Bessel{-\alpha}{p'}(\Omega)\right]' \quad \text{if } \alpha < 0.
\end{equation}
Moreover, it was shown in \cite[Lemma 9.1]{MR1658089} that for $\alpha>0$ and $1<p<\infty$, the gradient operator
\begin{equation}\label{eq:gradient-bounded}
\nabla : \Bessel{\alpha}{p}(\Omega)\rightarrow \Bessel{\alpha-1}{p}(\Omega)^n
\end{equation} is well-defined and bounded.  These observations enable us to define a generalized normal trace of a vector field $F$ under some additional regularity assumption.
\begin{prop}\label{prop:normal-trace}
Let  $(\alpha,p)$ and $(\beta,q)$ satisfy
\begin{equation}\label{eq:normal-trace-1}
\begin{gathered}
1<p<\infty,\quad -\frac{1}{p'}<\alpha<\frac{1}{p},\\
\alpha \leq \beta,\quad  0<\frac{1}{q}<\beta\leq 1,\quad \text{and }\quad \frac{1}{p}-\frac{\alpha}{n}\geq \frac{1}{q}-\frac{\beta}{n}.
\end{gathered}
\end{equation}
Assume that  $F\in \Bessel{\alpha}{p}(\Omega)^n$ and $\Div F \in \Bessel{\beta-1}{q}(\Omega)$. Then there exists a unique $\gamma_{\nu}(F) \in \bes{\alpha-1/p}{p}(\partial\Omega)$ such that
\begin{equation}\label{eq:normal-trace-identity}
      \action{\gamma_{\nu}(F),\Tr \phi} =  \action{F,\nabla \phi}+\action{\Div F,\phi}\quad \text{for all } \phi \in \Bessel{1-\alpha}{p'}(\Omega).
\end{equation}
Moreover, we have
\[     \norm{\gamma_{\nu}(F)}{\bes{\alpha-1/p}{p}(\partial\Omega)} \leq C ( \norm{F}{\Bessel{\alpha}{p}(\Omega)} + \norm{\Div F}{\Bessel{\beta-1}{q}(\Omega)})  \]
for some constant $C=C(n,\alpha,\beta,p,q,\Omega)>0$. In addition, if $F\in \Bessel{\beta}{q}(\Omega)^n$,    then
\[
 \Tr F \cdot \nu \in \bes{\alpha-1/p}{p}(\partial\Omega) \quad\mbox{and}\quad \gamma_{\nu}(F) = \Tr F \cdot \nu .
\]
\end{prop}
\begin{proof}
Let $\phi \in \Bessel{1-\alpha}{p'}(\Omega)$ be given. Then by \eqref{eq:pairing-1}, \eqref{eq:pairing-2}, and \eqref{eq:gradient-bounded}, the pairing $\action{F,\nabla \phi}$ is well-defined and
\[     \left|\action{F,\nabla \phi} \right| \leq C\norm{F}{\Bessel{\alpha}{p}(\Omega)} \norm{\phi}{\Bessel{1-\alpha}{p'}(\Omega)}\quad  \]
for some constant $C=C(n,\alpha,p,\Omega)>0$. Since \eqref{eq:normal-trace-1} holds,   it follows from Theorem \ref{thm:Sobolev-embedding} and \eqref{eq:zero-equal} that
\begin{equation}\label{eq:embeeding-normal-trace}
\Bessel{1-\alpha}{p'}(\Omega) \hookrightarrow   \Bessel{1-\beta}{q'}(\Omega) = \oBessel{1-\beta}{q'}(\Omega).
\end{equation}
Since $\Div F \in \Bessel{\beta-1}{q}(\Omega)$,  we thus have
\[   \left|\action{\Div F,\phi}\right|\leq \norm{\Div F}{\Bessel{\beta-1}{q}(\Omega)} \norm{\phi}{\Bessel{1-\beta}{q'}(\Omega)} \leq C \norm{\Div F}{\Bessel{\beta-1}{q}(\Omega)} \norm{\phi}{\Bessel{1-\alpha}{p'}(\Omega)}.\]

We now define $\gamma_{\nu} (F) : \bes{1/p-\alpha}{p'}(\partial\Omega) \rightarrow \mathbb{R}$ by
\begin{equation*}
\action{\gamma_{\nu}(F),\eta} = \action{F,\nabla (\mathcal{E}\eta)}+\action{\Div F,\mathcal{E}\eta}\quad \text{for all } \eta \in \bes{1/p-\alpha}{p'}(\partial\Omega),
\end{equation*}
where $\mathcal{E}$ is the extension operator given in Theorem \ref{thm:trace-theorem}. By \eqref{eq:embeeding-normal-trace} and Theorem \ref{thm:trace-theorem}, there exists a constant $C=C(n,\alpha,\beta,p,q,\Omega)>0$ such that
\begin{align*}
\left|\action{\gamma_{\nu}(F),\eta} \right| &\leq C \left(\norm{F}{\Bessel{\alpha}{p}(\Omega)}+ \norm{\Div F}{\Bessel{\beta-1}{q}(\Omega)} \right)\norm{\mathcal{E}\eta}{\Bessel{1-\alpha}{p'}(\Omega)}\\
&\leq C \left(\norm{F}{\Bessel{\alpha}{p}(\Omega)}+ \norm{\Div F}{\Bessel{\beta-1}{q}(\Omega)} \right)\norm{\eta}{\bes{1/p-\alpha}{p'}(\partial\Omega)}
\end{align*}
for all $\eta \in \bes{1/p-\alpha}{p'}(\partial\Omega)$.   It follows that  $\gamma_{\nu}(F)\in \bes{\alpha-1/p}{p}(\partial\Omega)$.

Next, we prove that $\gamma_{\nu}(F)$ is the unique element in $\bes{\alpha-1/p}{p}(\partial\Omega)$ satisfying \eqref{eq:normal-trace-identity}. Let $\phi \in \Bessel{1-\alpha}{p'}(\Omega)$ be fixed. Then by Theorem \ref{thm:trace-theorem}, $\Tr \phi \in \bes{1/p-\alpha}{p'}(\partial\Omega)$ and $\phi -\mathcal{E}(\Tr \phi)\in \oBessel{1-\alpha}{p'}(\Omega)$. Hence by the definition of $\Div F$, we have
\[    \action{F,\nabla(\phi-\mathcal{E}(\Tr \phi))} + \action{\Div F, \nabla (\phi-\mathcal{E}(\Tr \phi))}=0.\]
This implies that
\[   \action{\gamma_{\nu}(F),\Tr \phi} = \action{F,\nabla(\mathcal{E}(\Tr \phi))} + \action{\Div F,\mathcal{E}(\Tr \phi)} = \action{F,\nabla \phi} + \action{\Div F, \phi},  \]
which shows that $\gamma_{\nu}(F)$ satisfies identity \eqref{eq:normal-trace-identity}.  To show the uniqueness part, suppose that  $g\in \bes{\alpha-1/p}{p}(\partial\Omega)$  satisfies
\[    \action{g,\Tr \phi}=0\quad \text{for all } \phi \in \Bessel{1-\alpha}{p'}(\Omega). \]  Then for every $h\in \bes{1/p-\alpha}{p'}(\partial\Omega)$, we have
\[  \mathcal{E}h \in \Bessel{1-\alpha}{p'}(\Omega)\quad \text{and so}\quad  \action{g,h}=\action{g,\Tr(\mathcal{E}h)}=0.\]
This proves the uniqueness part.

Suppose in addition that $F\in \Bessel{\beta}{q}(\Omega)^n$.  Let us take
\[   r=\begin{cases}
          \frac{n-1}{n/q-\beta} & \text{if } \beta<{n}/{q}, \\
          \frac{n-1}{n/p-\alpha}+1 &  \text{if } \beta\geq {n}/{q}.
\end{cases}\]
Then by \eqref{eq:normal-trace-1} and Theorem \ref{thm:Sobolev-embedding}, we have
\begin{equation}\label{eq:normal-trace-more-3}
\bes{\beta-1/q}{q}(\partial\Omega) \hookrightarrow \Leb{r}(\partial\Omega)\quad \text{and }\quad \bes{1/p-\alpha}{p'} (\partial\Omega)\hookrightarrow \Leb{r'}(\partial\Omega).
\end{equation}
Hence  by H\"older's inequality, \eqref{eq:normal-trace-more-3}, and Theorem \ref{thm:trace-theorem}, there exists a constant $C=C(n,\alpha,\beta,p,q,\Omega)>0$ such that
\begin{equation}\label{eq:normal-trace-more-1}
  \left|\int_{\partial \Omega} ( \Tr F \cdot \nu) \eta \myd{\sigma} \right| \leq C \norm{F}{\Bessel{\beta}{q}(\Omega)} \norm{\eta}{\bes{1/p-\alpha}{p'}(\partial\Omega)}
\end{equation}
for all $\eta \in \bes{1/p-\alpha}{p'}(\partial\Omega)$.  This shows that $\Tr F \cdot \nu \in \bes{\alpha-1/p}{p}(\partial\Omega)$.  Since $\nabla:\Bessel{\beta}{q}(\Omega)\rightarrow \Bessel{\beta-1}{q}({\Omega})^n$ is bounded, $\Bessel{1-\alpha}{p'}(\Omega)\hookrightarrow \oBessel{1-\beta}{q'}(\Omega)$,  and $C^\infty(\overline{\Omega})^n$ is dense in $\Bessel{\beta}{q}(\Omega)^n$,  a standard density argument enables us to deduce from \eqref{eq:integration-by-parts-formula}  that
\begin{equation}\label{eq:normal-trace-more-2}
     \int_{\partial \Omega} ( \Tr F \cdot \nu) \Tr \phi\myd{\sigma}  = \action{F,\nabla \phi} + \action{\Div F, \phi}
\end{equation}
for all $\phi \in \Bessel{1-\alpha}{p'}(\Omega)$. From \eqref{eq:normal-trace-identity} and \eqref{eq:normal-trace-more-2}, we get
\[   \action{\Tr F \cdot \nu -\gamma_{\nu}(F),\Tr \phi} =0\quad \text{for all } \phi \in \Bessel{1-\alpha}{p'}(\Omega). \]
Hence it follows that $\Tr F \cdot \nu=\gamma_{\nu}(F)$. This completes the proof of Proposition \ref{prop:normal-trace}.
\end{proof}

\subsection{The Poisson equation}
We first consider the  Dirichlet problem for the Poisson equation
\begin{equation}\label{eq:D-1-basic}
\left\{\begin{alignedat}{2}
-\triangle u& =f & \quad & \text{in }\Omega,\\
 u   & =u_{\Di} & \quad & \text{on }\partial\Omega.
\end{alignedat}\right.
\end{equation}
Here   the boundary condition is satisfied in the sense of nontangential convergence or  trace which will be  described in Theorems \ref{thm:L2-nontangential}, \ref{thm:Dirichlet-Poisson}, and \ref{thm:Dirichlet-Poisson-2}, respectively.

Let $\left\{ \gamma\left(x\right):x\in\partial\Omega\right\} $ be a regular family of cones associated with the Lipschitz domain $\Omega$ (see \cite[Section 0]{MR769382}) and let $u$ be a function on $\Omega$.
The nontangential maximal function of $u$ is defined by
\[
u^{*}(x)=\sup_{y\in\gamma(x)}\left|u(y)\right|\quad\text{for all }x\in\partial\Omega.
\]
If there is a function $g$ on $\partial\Omega$ such that
\[
\lim_{z\rightarrow x,z\in\gamma(x)}u(z)=g(x)\quad\text{for a.e. }x\in\partial\Omega,
\]
we write
\[
u\rightarrow g\quad\text{nontangentially a.e. on }\partial\Omega.
\]

The following proposition shows, in particular,  that harmonic functions in $\Bessel{1/2}{2}(\Omega)$ have nontangential limits (see \cite[Corollary 5.5]{MR1331981}).
\begin{prop}
\label{prop:JK-nontangential-maximal-function}Suppose that $u$ is
a harmonic function in $\Omega$. Then $u^{*}\in\Leb 2(\partial\Omega)$
if and only if $u\in\Bessel{1/2}2(\Omega)$. In each case, there exists a function $g \in\Leb 2(\partial\Omega)$
such that $u\rightarrow g$ nontangentially a.e. on $\partial\Omega$.
\end{prop}

Proposition \ref{prop:JK-nontangential-maximal-function} leads us to introduce a function space
\[      \haSob{1/2}{2}(\Omega) = \left\{ u \in \Bessel{1/2}{2}(\Omega) :  \triangle u =0\text{ in } \Omega \right\},  \]
which is a closed subspace of $\Bessel{1/2}{2}(\Omega)$.

The following theorem, which originated from Dahlberg \cite{MR466593,MR562447}, summarizes the classical  solvability and regularity results for the  Dirichlet problem \eqref{eq:D-1-basic} with $f=0$ and $u_{\Di}  \in \Leb{2}(\partial\Omega)$ (see  \cite[Theorems 5.3, 5.15]{MR1331981} for a convenient reference).

\begin{thm}\label{thm:L2-nontangential} For every  $u_{\Di} \in \Leb{2}(\partial\Omega)$, there exists a unique $u\in \haSob{1/2}{2}(\Omega)$ such that   $u\rightarrow u_{\Di }$ nontangentially a.e. on $\partial\Omega$. Moreover, we have
\[
\norm{u^{*}}{\Leb 2(\partial\Omega)}\le C(n,\Omega)\norm{u_{\Di}}{\Leb 2(\partial\Omega)}.
\]
In addition,  if  $u_{\Di}\in\Leb q(\partial\Omega)$ and $2\leq q\leq \infty$, then
\begin{equation*}
u\in\Bessel{1/q}q(\Omega)\quad\text{and }\quad  \norm u{\Bessel{1/q}q(\Omega)}\le C\left(n,q,\Omega\right)\norm{u_{\Di}}{\Leb q(\partial\Omega)}.
\end{equation*}
\end{thm} 

Jerison-Kenig \cite[Theorems 1.1, 1.3]{MR1331981} obtained the following optimal solvability results in   $\Bessel{\alpha}p(\Omega)$ for the inhomogeneous Dirichlet problem \eqref{eq:D-1-basic}. 
\begin{thm}\label{thm:Dirichlet-Poisson} Let $\Omega$ be a bounded Lipschitz domain in $\mathbb{R}^n$, $n\geq 3$. There is a number $\varepsilon=\varepsilon(\Omega)\in(0,1]$
such that if  $(\alpha,p)$ satisfies one of the following conditions
\begin{align*}
\mathrm{(i)}\quad & 1-\varepsilon\le2/p\leq1+\varepsilon\quad\text{and}\quad1/p<\alpha<1+1/p;\\
\mathrm{(ii)}\quad & 1+\varepsilon<2/p<2\quad\text{and}\quad3/p-1-\varepsilon<\alpha<1+1/p;\\
\mathrm{(iii)}\quad & 0<2/p<1-\varepsilon\quad\text{and}\quad1/p<\alpha<3/p+\varepsilon,
\end{align*}
then for every $f\in\Bessel{\alpha-2}p(\Omega)$ and $u_{\Di}\in\bes{\alpha-1/p}p(\partial\Omega)$,
there exists a unique function  $u\in\Bessel{\alpha}p(\Omega)$
such that
\[
\left\{\begin{alignedat}{2}
-\triangle u & =f & \quad & \text{in }\Omega,\\
\Tr   u  & =u_{\Di} & \quad & \text{on }\partial\Omega.
\end{alignedat}\right.
\]
Moreover, this function $u$ satisfies
\[
\norm u{\Bessel{\alpha}p(\Omega)}\le C\left(\norm f{\Bessel{\alpha-2}p(\Omega)}+\norm{u_{\Di}}{\bes{\alpha-1/p} p(\partial\Omega)}\right)
\]
for some constant $C=C(n,\alpha,p,\Omega)$.  If $\Omega$ is a $C^{1}$-domain, then the constant $\varepsilon$ may be taken one.  
\end{thm}

\begin{thm}\label{thm:Dirichlet-Poisson-2} Let $\Omega$ be a bounded Lipschitz domain in $\mathbb{R}^2$. There is a number $\varepsilon=\varepsilon(\Omega)\in(0,1]$
such that if  $(\alpha,p)$ satisfies one of the following conditions
\begin{align*}
\mathrm{(i)}\quad & 1-\varepsilon\le2/p\leq1+\varepsilon\quad\text{and}\quad1/p<\alpha<1+1/p;\\
\mathrm{(ii)}\quad & 1+\varepsilon<2/p<2\quad\text{and}\quad {2}/{p}-{(1+\varepsilon)}/{2}<\alpha<1+1/p;\\
\mathrm{(iii)}\quad & 0<2/p<1-\varepsilon\quad\text{and}\quad1/p<\alpha<2/p+(1+\varepsilon)/2,
\end{align*} 
then for every $f\in\Bessel{\alpha-2}p(\Omega)$ and $u_{\Di}\in\bes{\alpha-1/p}p(\partial\Omega)$,
there exists a unique function  $u\in\Bessel{\alpha}p(\Omega)$
such that
\[
\left\{\begin{alignedat}{2}
-\triangle u & =f & \quad & \text{in }\Omega,\\
\Tr   u  & =u_{\Di} & \quad & \text{on }\partial\Omega.
\end{alignedat}\right.
\]
Moreover, this function $u$ satisfies
\[
\norm u{\Bessel{\alpha}p(\Omega)}\le C\left(\norm f{\Bessel{\alpha-2}p(\Omega)}+\norm{u_{\Di}}{\bes{\alpha-1/p} p(\partial\Omega)}\right)
\]
for some constant $C=C(\alpha,p,\Omega)$.  If $\Omega$ is a $C^{1}$-domain, then the constant $\varepsilon$ may be taken one.  
\end{thm}

\begin{defn}\label{defn:admissible}
We denote by $\mathscr{A}$ the set of all pairs $(\alpha,p)$ that
satisfy one of the conditions (i), (ii), and (iii) in Theorem  \ref{thm:Dirichlet-Poisson} $(n\geq 3)$ or Theorem \ref{thm:Dirichlet-Poisson-2} $(n=2)$.
\end{defn}

To illustrate   $\mathscr{A}$, we introduce
\[     \tilde{\mathscr{A}} = \left \{ \left(\alpha,\frac{1}{p} \right) : (\alpha,p) \in \mathscr{A}  \right\}.\]
See Figure \ref{fig:admissible} for the set $\tilde{\mathscr{A}}$ in the $\alpha p^{-1}$-plane.  Observe that $\tilde{\mathscr{A}}$ is   symmetric with respect to $(1,1/2)$; hence 
\[ (\alpha,p)\in \mathscr{A}\quad \text{if and only if } \quad (2-\alpha,p')\in \mathscr{A}.\]
\begin{figure}
 \begin{subfigure}[b]{0.8\textwidth}
         \centering
         \includegraphics[width=\textwidth,page=2]{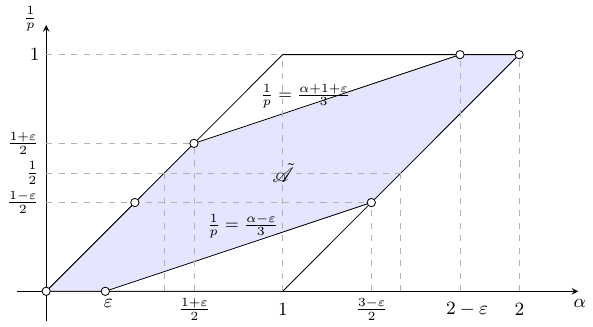} 
         \caption{The case $n\geq 3$}
         \label{fig:y equals x}
     \end{subfigure} 
     
      \begin{subfigure}[b]{0.8\textwidth}
         \centering
         \includegraphics[width=\textwidth,page=4]{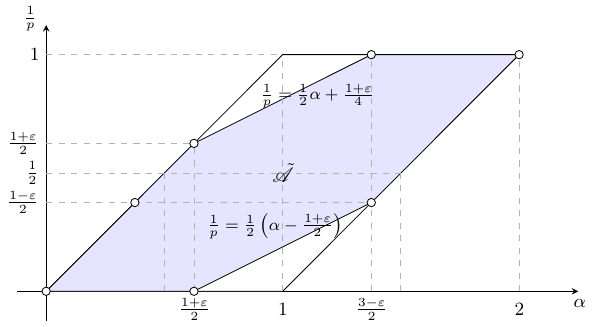} 
         \caption{The case $n=2$}
         \label{fig:y equals x}
     \end{subfigure} 

\caption{ \label{fig:admissible}}
\end{figure}
Next, let us consider the Neumann problem  for the  Poisson equation
\begin{equation}\label{eq:N-1-basic}
\left\{\begin{alignedat}{2}
-\triangle u& =f & \quad & \text{in }\Omega,\\
\nabla u \cdot \nu & =u_{\Neu} & \quad & \text{on }\partial\Omega.
\end{alignedat}\right.
\end{equation}

A standard  weak formulation of \eqref{eq:N-1-basic} is to find $u$ satisfying
\begin{equation*}
    \int_\Omega \nabla u \cdot \nabla \phi \myd{x} = \int_{\Omega} f \phi \myd{x} + \int_{\partial\Omega} u_{\Neu} \phi \myd{x}\quad \text{for all } \phi \in C^\infty(\overline{\Omega})
\end{equation*}
provided that the data $f$ and $u_{\Neu}$ are sufficiently regular. Note that for $1/p<\alpha<1+1/p$, the pairing between $\Bessel{\alpha-1}{p}(\Omega)$ and $\Bessel{1-\alpha}{p'}(\Omega)$ is well-defined by \eqref{eq:pairing-1} and \eqref{eq:pairing-2}.  So the pairing $\action{\nabla u,\nabla \phi}$ is well-defined for all $u\in \Bessel{\alpha}{p}(\Omega)$ and $\phi \in \Bessel{2-\alpha}{p'}(\Omega)$.

The following theorem is due to Fabes-Mendez-Mitrea \cite{MR1658089}  and Mitrea \cite{MR1883390}. 
\begin{thm}\label{thm:Neumann-Poisson}
Let $(\alpha,p)\in \mathscr{A}$. Then for every  $f\in \oBessel{\alpha-2}{p}(\Omega)$ and $u_{\Neu} \in \bes{\alpha-1-1/p}{p}(\partial\Omega)$ satisfying the compatibility condition $\action{f,1}+\action{u_{\Neu},1}=0$,
there exists a unique (up to additive constants)  function $u\in \Bessel{\alpha}{p}(\Omega)$ such that
\begin{equation}\label{eq:Neumann-integral-equation}
\action{\nabla u,\nabla \phi} = \action{f,\phi}+\action{u_{\Neu},\Tr \phi} \quad \text{for all } \phi \in \Bessel{2-\alpha}{p'}(\Omega).
\end{equation}
 Moreover, this function $u$ satisfies
\[
\norm u{\Bessel{\alpha}p(\Omega)}\le C\left(\norm f{\oBessel{\alpha-2}p(\Omega)}+\norm{u_{\Neu}}{\bes{\alpha-1-1/p}{p}(\partial\Omega)}\right)
\]
for some constant $C=C(n,\alpha,p,\Omega)$.
\end{thm}

However,  Theorem \ref{thm:Neumann-Poisson}  does not always gurantee  solvability for  the Neumann problem \eqref{eq:N-1-basic} as shown in the following example from  Amrouche and Rodr\'\i{g}uez-Bellido \cite{MR2763035}.

\begin{example}\label{example:Neumann-fails}
Let $(\alpha,p)\in \mathscr{A}$ be fixed.   By H\"older's inequality and Theorem \ref{thm:trace-theorem} (ii), we have
\begin{equation}\label{eq:counter-example}
\left|\int_{\partial \Omega} \Tr \phi \myd{\sigma} \right|\leq C \norm{\Tr \phi}{\bes{1+1/p-\alpha}{p'}(\partial\Omega)} \leq C \norm{\phi}{\Bessel{2-\alpha}{p'}(\Omega)}
\end{equation}
for all $\phi \in \Bessel{2-\alpha}{p'}(\Omega)$.
Define a linear functional $f$ by
\[     \action{f,\phi} = \int_{\partial\Omega} \Tr \phi \myd\sigma\quad \text{for all } \phi \in \Bessel{2-\alpha}{p'}(\Omega). \]
Then  by \eqref{eq:counter-example}, $f\in \oBessel{\alpha-2}{p}(\Omega)$.   Choose any $u_{\Neu} \in \bes{\alpha-1-1/p}{p}(\partial\Omega)$ satisfying
\[   \action{f,1}+\action{u_{\Neu},1}=0. \]
By Theorem \ref{thm:Neumann-Poisson}, there exists a function $u\in \Bessel{\alpha}{p}(\Omega)$ satisfying \eqref{eq:Neumann-integral-equation}. However, we  prove that $u$ satisfies
\begin{equation}\label{eq:Neumann-counterexample-2}
\left\{\begin{alignedat}{2}
-\triangle u & =0 & \quad & \text{in }\Omega,\\
\gamma_{\nu} (\nabla u) & =1+u_{\Neu} & \quad & \text{on }\partial\Omega,
\end{alignedat}\right.
\end{equation}
where $\gamma_{\nu}$ is the generalized normal trace operator introduced in Proposition \ref{prop:normal-trace}.

Since $\action{f,\phi}=\action{1,\Tr \phi}$ for all $\phi \in \Bessel{2-\alpha}{p'}(\Omega)$ and $u$ satisfies \eqref{eq:Neumann-integral-equation}, we have
\[ \action{\nabla u,\nabla \phi} =0\quad \text{for all } \phi \in C_0^\infty(\Omega),\]
which shows that $\triangle u =0$ in $\Omega$.
Choose any $(\beta,q)$ satisfying
\[  \alpha \leq \beta,\quad 0<\frac{1}{q}<\beta-1\leq 1,\quad \text{and }\quad   \frac{1}{p}-\frac{\alpha}{n}\geq \frac{1}{q}-\frac{\beta}{n}.  \]
   Since $\nabla u \in \Bessel{\alpha-1}{p}(\Omega)$ and $\Div(\nabla u)=\triangle u=0 \in \Bessel{\beta-2}{q}(\Omega)$, it follows from Proposition \ref{prop:normal-trace} that   $\gamma_{\nu}(\nabla u) \in \bes{\alpha-1-1/p}{p}(\partial\Omega)$  and
\[    \action{\gamma_{\nu}(\nabla u),\Tr \phi} = \action{\nabla u,\nabla \phi}=\action{1+u_{\Neu},\Tr \phi} \]
for all $\phi \in \Bessel{1-\alpha}{p'}(\Omega)$. Finally, since $\Tr: \Bessel{1-\alpha}{p'}(\Omega)\rightarrow \bes{1/p-\alpha}{p'}(\partial\Omega)$ is surjective,  we get
\[ \gamma_{\nu}(\nabla u) =1+u_{\Neu}, \]
which proves that $u$ is a solution of the problem \eqref{eq:Neumann-counterexample-2}.
\end{example}

Example \ref{example:Neumann-fails} suggests that we need to assume more regularity on the data $f$ to gurantee a unique solvability result for the Neumann problem \eqref{eq:N-1-basic}.
\begin{thm}\label{thm:real-Neumann-easy-one}
Let $(\alpha,p)\in \mathscr{A}$ and assume that $(\beta,q)$ satisfies
\[  \alpha\leq \beta,\quad   0<\frac{1}{q}<\beta-1 \leq 1,   \quad\text{and}\quad\frac{1}{p}-\frac{\alpha}{n} \geq \frac{1}{q}-\frac{\beta}{n}. \]
Then for every $f\in \Bessel{\beta-2}{q}(\Omega)$ and $u_{\Neu} \in \bes{\alpha-1-1/p}{p}(\partial\Omega)$ satisfying the compatibility condition $ \action{f,1}+\action{u_{\Neu},1}=0$, there exists a unique (up to additive constants)   function $u\in \Bessel{\alpha}{p}(\Omega)$ such that
\begin{equation}\label{eq:N-4-basic}
\left\{\begin{alignedat}{2}
-\triangle u & =f & \quad & \text{in }\Omega,\\
\gamma_{\nu} (\nabla u) & =u_{\Neu} & \quad & \text{on }\partial\Omega.
\end{alignedat}\right.
\end{equation}
\end{thm}
\begin{proof}
Since $\Bessel{2-\alpha}{p'}(\Omega)\hookrightarrow \Bessel{2-\beta}{q'}(\Omega)=\oBessel{2-\beta}{q'}(\Omega)$ and $\Bessel{2-\alpha}{p'}(\Omega)$ is dense in $\Bessel{2-\beta}{q'}(\Omega)$, it follows that $\Bessel{\beta-2}{q}(\Omega)\hookrightarrow \oBessel{\alpha-2}{p}(\Omega)$. Hence  by Theorem \ref{thm:Neumann-Poisson}, there exists a  function $u\in \Bessel{\alpha}{p}(\Omega)$ such that
\begin{equation*}
\action{\nabla u,\nabla \phi} =\action{f,\phi} +\action{u_{\Neu},\Tr \phi}\quad \text{for all } \phi \in \Bessel{2-\alpha}{p'}(\Omega).
\end{equation*}
Since $f\in \Bessel{\beta-2}{q}(\Omega)$, we have
\[    \action{\nabla u,\nabla \phi} =\action{f,\phi} \quad \text{and }\quad    \left|\action{\nabla u,\nabla \phi}\right|\leq C \norm{f}{\Bessel{\beta-2}{q}(\Omega)} \norm{\phi}{\Bessel{2-\beta}{q'}(\Omega)} \]
for all $\phi \in C_0^\infty(\Omega)$. So $\Div(\nabla u) = \triangle u \in \Bessel{\beta-2}{q}(\Omega)$ and $-\triangle u =f$ in $\Omega$. Moreover it follows from Proposition \ref{prop:normal-trace} that  $\gamma_{\nu} (\nabla u) \in \bes{\alpha-1-1/p}{p}(\partial\Omega)$ and
\[   \action{\gamma_{\nu}(\nabla u),\Tr \phi} = -\action{f,\phi} + \action{\nabla u,\nabla \phi}=\action{u_{\Neu},\Tr \phi}   \]
for all $\phi \in \Bessel{2-\alpha}{p'}(\Omega)$.
Since $\Tr : \Bessel{2-\alpha}{p'}(\Omega)\rightarrow \bes{1+1/p-\alpha}{p'}(\partial\Omega)$ is surjective,  we conclude that  $\gamma_{\nu} (\nabla u) =u_{\Neu}$.

To prove the uniqueness part, let $u$ be a  solution to the problem \eqref{eq:N-4-basic} with $f=0$ and $u_{\Neu}=0$.  Since $\gamma_{\nu}(\nabla u)=0$  and $-\triangle u=0$ in $\Omega$, it follows that
\[    0=\action{u_{\Neu},\Tr \phi} = \action{\nabla u,\nabla \phi} + \action{\triangle u,\phi} =   \action{\nabla u,\nabla \phi} \]
for all $\phi \in \Bessel{2-\beta}{q'}(\Omega)$. Hence by Theorem  \ref{thm:Neumann-Poisson}, $u=c$ for some constant $c$. This completes the proof of Theorem \ref{thm:real-Neumann-easy-one}.
\end{proof}

It can be easily shown that for each $(\alpha,p)\in \mathscr{A}$ there always exist    pairs $(\beta,q)$ satisfying the condition of Theorem \ref{thm:real-Neumann-easy-one}, for instance,  by using a geometric interpretation  of the exponents in the $\alpha p^{-1}$-plane (see  Figure 2.1).

\subsection{Bilinear estimates}\label{subsec:bilinear}
In this subsection, we derive some bilinear estimates which will play a crucial role in this paper.
\begin{lem}
\label{lem:basic-estimates} Suppose that $\boldb\in\Leb n(\Omega)^{n}$, and let $(\alpha,p)$ and $\left(\beta,q\right)$ satisfy
\begin{equation}
0\le\alpha\le\beta\le2,\quad1<q\le p<\infty,\quad\text{and }\quad \frac{1}{q}-\frac{\beta}{n}=\frac{1}{p}-\frac{\alpha}{n}.\label{eq:bilinear-estimates-Sobolev}
\end{equation}
\begin{enumerate}[label={\textnormal{(\roman*)}},ref={\roman*},leftmargin=2em]
\item \label{enu:basic-estimates-i} Assume that
\begin{equation}\label{eq:B-set}
\beta\le1\quad \text{and}\quad \frac{\alpha}{n}<\frac{1}{p}<\frac{\alpha+n-1}{n}.
\end{equation}
Then for any $u\in\Bessel{\alpha}p(\Omega)$, we have
\[  u\boldb\in \Leb{1}(\Omega)^n\]
and
\begin{equation}\label{eq:Gerhardt-constant}
\int_\Omega (u\boldb ) \cdot \Phi \myd{x} \leq C\norm{\boldb}{\Leb{n}(\Omega)} \norm{u}{\Bessel{\alpha}{p}(\Omega)} \norm{\Phi}{\Bessel{1-\beta}{q'}(\Omega)}
\end{equation}
for all $ \Phi \in C^\infty(\overline{\Omega})^n$, where  $C=C\left(n,\alpha,\beta,p,q,\Omega\right)>0$.
\item \label{enu:basic-estimates-ii} Assume that
\begin{equation}\label{eq:B-prime-set}
\alpha\geq1\quad \text{and}\quad \frac{\alpha-1}{n}<\frac{1}{p}<\frac{\alpha+n-2}{n}.
\end{equation}
Then for any $v\in\Bessel{\alpha}p(\Omega)$, we have
\[ \boldb\cdot \nabla v\in \Leb{1}(\Omega)\]
and
\begin{equation*}
\int_\Omega (\boldb\cdot \nabla v)   \psi \myd{x} \leq C\norm{\boldb}{\Leb{n}(\Omega)} \norm{v}{\Bessel{\alpha}{p}(\Omega)} \norm{\psi}{\Bessel{2-\beta}{q'}(\Omega)}
\end{equation*}
for all $\psi \in C^\infty(\overline{\Omega})$, where  $C=C\left(n,\alpha,\beta,p,q,\Omega\right)>0$.
\end{enumerate}
\end{lem}

\begin{proof}
Assume that $(\alpha,p)$ satisfies \eqref{eq:B-set} and $u\in \Bessel{\alpha}{p}(\Omega)$. Then it follows from  Theorem \ref{thm:Sobolev-embedding} that $u\in \Leb{\frac{n}{n-1}}(\Omega)$.   Hence by H\"older's inequality, we have $ u\boldb \in \Leb{1}(\Omega)^n$.

Define $r$ and $s$ by
\[
\frac{1}{r}=\frac{1}{p}-\frac{\alpha}{n}\quad\text{and}\quad\frac{1}{s}=\frac{1}{q'}-\frac{1-\beta}{n}.
\]
Then
\begin{equation}\label{eq:basic-estimates-condition}
1<r,s<\infty\quad  \text{and}\quad  \frac{1}{r}+\frac{1}{s}+\frac{1}{n}=1.
\end{equation}
Let $\Phi\in C^{\infty}\left(\overline{\Omega}\right)^{n}$ be given. Then H\"older's
inequality and Theorem \ref{thm:Sobolev-embedding}   give
\begin{align*}
\int_{\Omega}\left(u\boldb\right)\cdot\Phi \myd{x} & \le\norm{\boldb}{\Leb n(\Omega)}\norm u{\Leb r(\Omega)}\norm{\Phi}{\Leb s(\Omega)}\\
 & \le C\norm{\boldb}{\Leb n(\Omega)}\norm u{\Bessel{\alpha}p(\Omega)}\norm{\Phi}{\Bessel{1-\beta}{q'}(\Omega)}
\end{align*}
for some constant $C=C\left(n,\alpha,\beta,p,q,\Omega\right)$. This
completes the proof of \eqref{enu:basic-estimates-i}. The assertion \eqref{enu:basic-estimates-ii}  immediately follows from (i) since $v\in\Bessel{\alpha}{p}(\Omega)$ implies $\nabla v \in \Bessel{\alpha-1}{p}(\Omega)^n$. This completes the proof of Lemma \ref{lem:basic-estimates}.
\end{proof}
\begin{rem*}
Suppose that  $\alpha \ge 0$ and $1<p<\infty$. If $w\in \Leb{1}(\Omega)$ satisfies
\[   \int_\Omega  w \phi \myd{x} \leq C \norm{\phi}{\Bessel{\alpha}{p}(\Omega)}\quad \text{for all } \phi \in C^\infty(\overline{\Omega}), \]
then the functional
\[   \phi \mapsto \int_\Omega w \phi \myd{x} \]
can be uniquely extended to a bounded linear functional on both $\Bessel{\alpha}{p}(\Omega)$ and $\oBessel{\alpha}{p}(\Omega)$, which we still denote by $w$.
\end{rem*}

Lemma \ref{lem:basic-estimates} and the remark  enable us to prove the following estimates which are inspired by Gerhardt's inequality in \cite{MR520820} (see also \cite{MR2506072,MR3328143,MR3623550,MR2846167}).
\begin{lem}
\label{lem:Gerhardt} Suppose that $\boldb\in\Leb n(\Omega)^n$.
\begin{enumerate}[label={\textnormal{(\roman*)}},ref={\roman*},leftmargin=2em]
\item \label{enu:Gerhardt-1}  Let $(\alpha,p)$ satisfy
\[   0\leq \alpha\leq1 \quad  \text{and }\quad  \frac{\alpha}{n}< \frac{1}{p}< \frac{\alpha+n-1}{n}.\]
Then for each $\varepsilon>0$, there is a constant $C_{\varepsilon}=C\left(\varepsilon,n,\alpha,p,\boldb,\Omega\right)>0$
such that
\[
\norm{u\boldb}{\oBessel{\alpha-1}p(\Omega)} + \norm{u\boldb}{\Bessel{\alpha-1}{p}(\Omega)} \le\varepsilon\norm u{\Bessel{\alpha}p(\Omega)}+C_{\varepsilon}\norm u{\Leb p(\Omega)}
\]
for all $u\in\Bessel{\alpha}p(\Omega)$.
\item \label{enu:Gerhardt-2} Let  $(\alpha,p)$ satisfy
\[ 1\le\alpha \leq 2\quad \text{and } \quad \frac{\alpha-1}{n}<\frac{1}{p}<\frac{\alpha+n-2}{n}.\]
Then for each $\varepsilon>0$, there is a constant $C_{\varepsilon}=C\left(\varepsilon,n,\alpha,p,\boldb,\Omega\right)>0$
such that
\[
\norm{\boldb\cdot\nabla v}{\oBessel{\alpha-2}p(\Omega)} + \norm{\boldb \cdot \nabla v}{\Bessel{\alpha-2}{p}(\Omega)} \le\varepsilon\norm v{\Bessel{\alpha}p(\Omega)}+C_{\varepsilon}\norm v{\Leb p(\Omega)}
\]
for all $v\in\Bessel{\alpha}p(\Omega)$.
\end{enumerate}
\end{lem}

\begin{proof}
(\ref{enu:Gerhardt-1}) Let $\varepsilon>0$ be given. Since $C_{0}^{\infty}(\Omega)^{n}$
is dense in $\Leb n(\Omega)^{n}$, there exists $\boldb_{\varepsilon}\in C_{0}^{\infty}(\Omega)^n$
such that $\norm{\boldb_{\varepsilon}-\boldb}{\Leb n(\Omega)}<\varepsilon/2C$,
where $C$ is the positive constant in \eqref{eq:Gerhardt-constant} with $\beta =\alpha$ and $q =p$.
Then by  Lemma \ref{lem:basic-estimates} \eqref{enu:basic-estimates-i} and its remark,
we have
\begin{align*}
\norm{u\boldb}{\oBessel{\alpha-1}p(\Omega)} & \le\norm{u\left(\boldb-\boldb_{\varepsilon}\right)}{\oBessel{\alpha-1}p(\Omega)}+\norm{u\boldb_{\varepsilon}}{\oBessel{\alpha-1}p(\Omega)}\\
 & \le \frac{\varepsilon}{2}\norm u{\Bessel{\alpha}p(\Omega)}+
 \norm{u\boldb_{\varepsilon}}{\oBessel{\alpha-1}p(\Omega)}
\end{align*}
and similarly
\[
\norm{u\boldb}{\Bessel{\alpha-1}p(\Omega)}  \le\frac{\varepsilon}{2}\norm u{\Bessel{\alpha}p(\Omega)}+
 \norm{u\boldb_{\varepsilon}}{\Bessel{\alpha-1}p(\Omega)} .
\]
Let $\Phi\in C^{\infty}\left(\overline{\Omega}\right)^{n}$. Then by H\"older's
inequality, we get
\begin{align*}
\int_\Omega  (u\boldb_{\varepsilon}) \cdot \Phi \myd{x}  & \le\norm{\boldb_{\varepsilon}}{\Leb{\infty}(\Omega)}\norm u{\Leb p(\Omega)}\norm{\Phi}{\Leb{p'}(\Omega)} \\
&\leq   \norm{\boldb_{\varepsilon}}{\Leb{\infty}(\Omega)} \norm{u}{\Leb{p}(\Omega)} \norm{\Phi}{\Bessel{1-\alpha}{p'}(\Omega)}
\end{align*}
and thus
\[
\norm{u\boldb_{\varepsilon}}{\oBessel{\alpha-1}p(\Omega)}\le C_{\varepsilon}\norm u{\Leb p(\Omega)}
\]
for some constant $C_{\varepsilon}=C\left(\varepsilon,n,\alpha,p,\boldb,\Omega\right)$.
The proof  of \eqref{enu:Gerhardt-2} is similar and so omitted. This completes the proof of Lemma \ref{lem:Gerhardt}.
\end{proof}
\begin{figure}

 \begin{subfigure}[b]{0.8\textwidth}
         \centering
         \includegraphics[width=\textwidth,page=1]{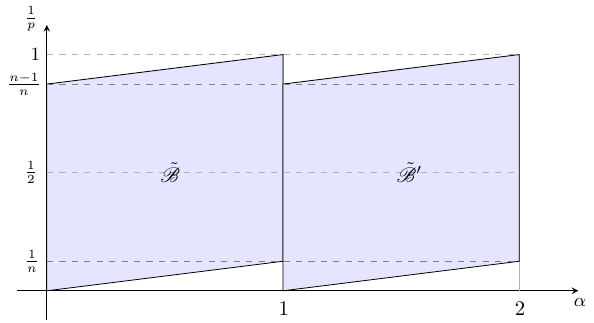} 
         \caption{The case  $n\geq 3$} 
     \end{subfigure}

 \begin{subfigure}[b]{0.8\textwidth}
         \centering
         \includegraphics[width=\textwidth,page=5]{Picture/diagram5} 
         \caption{The case $n=2$} 
     \end{subfigure}  
\caption{\label{fig:bilinear-estimates}}

\end{figure}
\begin{defn}\label{defn:admissible-B}
We denote by $\mathscr{B}$ the set of all pairs $(\alpha,p)$ that
satisfy the condition in Lemma \ref{lem:Gerhardt}
\eqref{enu:Gerhardt-1}. Similarly, $\mathscr{B}'$ is the set of
all pairs $(\alpha,p)$ that satisfy the condition in Lemma
\ref{lem:Gerhardt} \eqref{enu:Gerhardt-2}.
\end{defn}
To depict  these sets, we introduce
\[  \tilde{\mathscr{B}} = \left\{  \left(\alpha,\frac{1}{p} \right) : (\alpha,p) \in \mathscr{B} \right\}\quad \text{and}\quad \tilde{\mathscr{B}'} = \left\{   \left(\alpha,\frac{1}{p} \right) : (\alpha,p) \in \mathscr{B}' \right\}. \]
See Figure \ref{fig:bilinear-estimates} for the  sets $\tilde{\mathscr{B}}$ and $\tilde{\mathscr{B}}'$ in the $\alpha p^{-1}$-plane.   Note that  $\tilde{\mathscr{B}}'$ is the reflection of $\tilde{\mathscr{B}}$ with respect to $(1,1/2)$; hence
\[  (\alpha,p)\in \mathscr{B}\quad \text{if and only if}\quad  (2-\alpha,p')\in \mathscr{B}'. \]

Assume that $\boldb \in \Leb{n}(\Omega)^n$, and  let  $(\alpha,p) \in \mathscr{B}$ be fixed. Then by Lemma \ref{lem:basic-estimates} (i), the mapping
\[  (u,\phi)\mapsto   \action{\mathcal{B}_{\alpha,p}^{N}u,\phi} = -\action{u\boldb,\nabla \phi}\quad \text{for all } (u,\phi) \in \Bessel{\alpha}{p}(\Omega)\times\Bessel{2-\alpha}{p'}(\Omega)\]
defines a bounded linear operator $\mathcal{B}_{\alpha,p}^N$ from $\Bessel{\alpha}{p}(\Omega)$ to $\oBessel{\alpha-2}{p}(\Omega)$. The same lemma also shows that  the mapping $\mathcal{B}^{D}_{\alpha,p}$  defined by
\[   \mathcal{B}^{D}_{\alpha,p} u = \Div(u\boldb)\quad \text{for all } u\in \oBessel{\alpha}{p}(\Omega)  \]
is a bounded linear operator from $\oBessel{\alpha}{p}(\Omega)$ to $\Bessel{\alpha-2}{p}(\Omega)$. Similarly, it follows from Lemma \ref{lem:basic-estimates} (ii) that the mapping  $\mathcal{B}_{2-\alpha,p'}^{*}$  defined by
\[   \mathcal{B}_{2-\alpha,p'}^{*}v= -\boldb\cdot \nabla v\quad  \text{for all } v \in \Bessel{2-\alpha}{p'}(\Omega) \]
is  a bounded linear operator   from $\Bessel{2-\alpha}{p'}(\Omega)$ to $\oBessel{-\alpha}{p'}(\Omega)$. Note also that
\begin{equation}\label{eq:duality-B}
\action{\mathcal{B}_{\alpha,p}^{N}u,v} = \action{\mathcal{B}_{2-\alpha,p'}^*v,u}\quad \text{for all } (u,v) \in \Bessel{\alpha}{p}(\Omega)\times \Bessel{2-\alpha}{p'}(\Omega)
\end{equation}
 and
\begin{equation}\label{eq:duality-B-2}
\action{\mathcal{B}_{\alpha,p}^{D}u,v} = \action{\mathcal{B}_{2-\alpha,p'}^*v,u}\quad \text{for all } (u,v) \in \oBessel{\alpha}{p}(\Omega)\times \oBessel{2-\alpha}{p'}(\Omega).
\end{equation}
 Moreover, these operators are compact as shown below.

\begin{lem}\label{lem:compactness-of-operators} Suppose that $\boldb \in \Leb{n}(\Omega)^n$, and let $(\alpha,p)\in \mathscr{B}$.   Then the operators $\mathcal{B}_{\alpha,p}^{N}$, $\mathcal{B}_{\alpha,p}^{D}$, and $\mathcal{B}_{2-\alpha,p'}^{*}$ are compact.
\end{lem}
\begin{proof}
It was shown in \cite[Proposition 2.9]{MR1331981} that $\Bessel{\beta}{q}(\Omega)$ is reflexive  for $\beta>0$ and $1<q<\infty$. Hence to prove the compactness of $\mathcal{B}_{\alpha,p}^{N}$, it suffices to show that $\mathcal{B}_{\alpha,p}^N$ is completely continuous; that is, if $u_k \rightarrow u$ weakly in $\Bessel{\alpha}{p}(\Omega)$, then $\mathcal{B}_{\alpha,p}^Nu_{k}\rightarrow\mathcal{B}_{\alpha,p}^Nu$
strongly in $\oBessel{\alpha-2}p(\Omega)$.

Let $\varepsilon>0$ be given. By Lemma \ref{lem:Gerhardt} (i),  there is a constant $C_\varepsilon =C(\varepsilon,n,\alpha,p,\boldb,\Omega)>0$ such that
\begin{align*}
\left|\action{\mathcal{B}_{\alpha,p}^N u,\phi}\right| & =\left|\action{u\boldb,\nabla\phi}\right|\le\norm{u\boldb}{\oBessel{\alpha-1}p\left(\Omega\right)}\norm{\nabla\phi}{\Bessel{1-\alpha}{p'}\left(\Omega\right)}\\
 & \le\left(\varepsilon\norm u{\Bessel{\alpha}p\left(\Omega\right)}+C_{\varepsilon}\norm u{\Leb p\left(\Omega\right)}\right)\norm{\phi}{\Bessel{2-\alpha}{p'}\left(\Omega\right)}
\end{align*}
for all $(u,\phi) \in \Bessel{\alpha}{p}(\Omega)\times \Bessel{2-\alpha}{p'}(\Omega)$. Hence it follows that
\begin{equation}\label{eq:compactness-of-operator}
\norm{\mathcal{B}_{\alpha,p}^N u}{\oBessel{\alpha-2}{p}(\Omega)} \leq \varepsilon \norm u{\Bessel{\alpha}p\left(\Omega\right)}+C_{\varepsilon}\norm u{\Leb p\left(\Omega\right)}\quad \text{for all } u \in \Bessel{\alpha}{p}(\Omega). \end{equation}
 Suppose that $u_k \rightarrow u$ weakly in $\Bessel{\alpha}{p}(\Omega)$ and  $\norm u{\Bessel{\alpha}p(\Omega)}\le M=\sup_{k}\norm{u_{k}}{\Bessel{\alpha}p(\Omega)}<\infty$.  Then by Theorem \ref{thm:Sobolev-embedding} (ii), we have
\[
 u_{k}\rightarrow u\quad\text{strongly in }\Leb p(\Omega).
\]
Thus by \eqref{eq:compactness-of-operator}, we get
\[
\limsup_{k\rightarrow\infty}\norm{\mathcal{B}_{\alpha,p}^Nu_{k}-\mathcal{B}_{\alpha,p}^Nu}{\oBessel{\alpha-2}p(\Omega)}\le2M\varepsilon.
\]
Since $\varepsilon>0$ was arbitrary chosen, it follows that $\mathcal{B}_{\alpha,p}^Nu_{k}\rightarrow\mathcal{B}_{\alpha,p}^Nu$
strongly in $\oBessel{\alpha-2}p(\Omega)$, which proves that $\mathcal{B}_{\alpha,p}^N$ is completely continuous. Since the proofs for  $\mathcal{B}_{\alpha,p}^{D}$ and $\mathcal{B}_{2-\alpha,p'}^{*}$ are similar, we omit their proofs.    This completes the proof of Lemma \ref{lem:compactness-of-operators}.
\end{proof}

\section{Main results}\label{sec:main-results}
We shall assume throughout the rest of the paper that
\[\boldb \in \Leb{n}(\Omega)^n. \]
Having   introduced the sets $\mathscr{A}$ and $\mathscr{B}$ in Section \ref{sec:prelim}, we are ready to state the main results of this paper.

\subsection{The Dirichlet problems}
Our first result is concerned with unique solvability for the Dirichlet problems \eqref{eq:D-1} and \eqref{eq:D-2} with boundary data in $\bes{\alpha}{p}(\partial\Omega)$.

\begin{thm}\label{thm:Dirichlet-problems}  Let $(\alpha,p)\in \mathscr{A}\cap \mathscr{B}$.
\begin{enumerate}[label={\textnormal{(\roman*)}},ref={\roman*},leftmargin=2em]
\item \label{enu:solvability-Bessel-sp-i} For every $f\in\Bessel{\alpha-2}p(\Omega)$ and $u_{\Di}\in\bes{\alpha-1/p}p(\partial\Omega)$,
there exists a unique solution $u\in\Bessel{\alpha}p(\Omega)$
of \eqref{eq:D-1}. Moreover, this solution $u$ satisfies
\[
\norm u{\Bessel{\alpha}p(\Omega)}\le C\left(\norm f{\Bessel{\alpha-2}p(\Omega)}+\norm{u_{\Di}}{\bes{\alpha-1/p}p(\partial\Omega)}\right)
\]
for some constant $C=C\left(n,\alpha,p,\boldb,\Omega\right)$.
\item \label{enu:solvability-Bessel-sp-ii} For every $g\in\Bessel{-\alpha}{p'}(\Omega)$ and $v_{\Di}\in\bes{1+1/p-\alpha}{p'}(\partial\Omega)$,
there exists a unique solution $v\in\Bessel{2-\alpha}{p'}(\Omega)$
of \eqref{eq:D-2} . Moreover, this solution $v$ satisfies
\[
\norm v{\Bessel{2-\alpha}{p'}(\Omega)}\le C\left(\norm g{\Bessel{-\alpha}{p'}(\Omega)}+\norm{v_{\Di}}{\bes{1+1/p-\alpha}{p'}(\partial\Omega)}\right)
\]
for some constant $C=C\left(n,\alpha,p,\boldb,\Omega\right)$.
\end{enumerate}
\end{thm}

\begin{rem*}\leavevmode
\begin{enumerate}
\item[(i)] For the special case when $n\geq 3$ and $(\alpha,p)=(1,2)$, Theorem \ref{thm:Dirichlet-problems} was already shown by Droniou \cite{MR1908676}. When $\Omega$ is a bounded $C^1$-domain in $\mathbb{R}^n$, $\Bessel{1}{p}$-results were established by Kim-Kim \cite{MR3328143} for $n\geq 3$ and Kwon \cite{K21} for $n=2$.
\item[(ii)] Theorem \ref{thm:Dirichlet-problems} extends the previous results in \cite{MR1908676,MR3328143,K21} as well as the classical work of  Jerison-Kenig \cite{MR1331981}.
\end{enumerate}  
\end{rem*}

Next, we consider unique solvability for the  Dirichlet problem \eqref{eq:D-1} with boundary data in $\Leb{2}(\partial\Omega)$.  Let $\left(\alpha,p\right)\in\mathscr{A}\cap\mathscr{B}$ be fixed. Suppose that $f\in \Bessel{\alpha-2}{p}(\Omega)$ and $u_{\Di} \in \Leb{2}(\partial\Omega)$. By Theorem \ref{thm:L2-nontangential}, there exists a unique function   $u_1 \in \mathcal{H}^2_{1/2}(\Omega)$
such that $u_1 \rightarrow u_{\Di}$ nontangentially a.e. on $\partial\Omega$.  Let us consider the following problem:
\begin{equation}\label{eq:perturbed-L2-data}
\left\{\begin{alignedat}{2}
-\triangle u_{2}+\Div\left(u_{2}\boldb\right)&=-\Div\left(u_{1}\boldb\right) &&\quad \text{in }\Omega,\\
u_{2}&=0  &&\quad \text{on }\partial\Omega.
\end{alignedat}\right.
\end{equation}
By Theorem \ref{thm:Sobolev-embedding} and Lemma \ref{lem:basic-estimates},
\[ u_1 \in \Leb{\frac{2n}{n-1}}(\Omega)\quad \text{and so}\quad \Div(u_1\boldb) \in \Bessel{-1}{\frac{2n}{n+1}}(\Omega).\]
It is easy to check that  $\left(1,{2n}/{(n+1)} \right) \in \mathscr{A}\cap\mathscr{B}$. Hence by  Theorem \ref{thm:Dirichlet-problems} \eqref{enu:solvability-Bessel-sp-i},  there exists a unique solution $u_{2}\in \oBessel 1{\frac{2n}{n+1}}(\Omega)$ of \eqref{eq:perturbed-L2-data}.  The same theorem also shows that  there exists a unique solution $u_3 \in \oBessel{\alpha}{p}(\Omega)$ of the problem \eqref{eq:D-1} with trivial boundary data. Define $u=u_1+u_2+u_3$. Then
\[  u\in  \haSob{1/2}2(\Omega)+\oBessel 1{\frac{2n}{n+1}}(\Omega)+\oBessel{\alpha}p(\Omega) \]
and
\[   -\triangle u +\Div(u\boldb) =f\quad \text{in } \Omega. \]

To proceed further, we need the following lemma which will be proved in  Section \ref{subsec:L2-boundary}.
\begin{lem}\label{lem:sumspace-decomposition}
Let $(\alpha,p) \in \mathscr{A}$ and $0<\alpha \leq 1$. Then
\[    \haSob{1/2}{2}(\Omega)\cap \left(\oBessel{1}{\frac{2n}{n+1}}(\Omega)+\oBessel{\alpha}{p}(\Omega) \right) =\{0\}. \]
\end{lem}
Lemma \ref{lem:sumspace-decomposition} motivates us to introduce  the following definition.
\begin{defn}
For  $\left(\alpha,p\right)\in\mathscr{A}$ and $0<\alpha \leq 1$, we denote by $D_{\alpha}^{p}(\Omega)$
the Banach space
\[
\haSob{1/2}{2}(\Omega) \oplus \left(\oBessel 1{\frac{2n}{n+1}}(\Omega)+\oBessel{\alpha}p(\Omega)\right)
\]
equipped with the natural norm. The projection operator of $D_{\alpha}^{p}(\Omega)$
onto $\haSob{1/2}2(\Omega)$ is  denoted by $\mathbb{P}_{\alpha,p}$
or simply by $\mathbb{P}$.
\end{defn}

Now we are ready to state unique solvability and regularity results  for the Dirichlet problem \eqref{eq:D-1} with boundary data in $\Leb 2(\partial\Omega)$.
\begin{thm}
\label{thm:L2-boundary}Let $(\alpha,p)\in\mathscr{A}\cap\mathscr{B}$. For every $f\in\Bessel{\alpha-2}p(\Omega)$ and $u_{\Di} \in \Leb{2}(\partial\Omega)$, there exists a unique function  $u\in D_{\alpha}^{p}(\Omega)$ such that
\begin{align*}
\textnormal{(a)}\quad &\displaystyle -\int_{\Omega}u\left(\triangle\phi+\boldb\cdot\nabla\phi\right)\myd{x}=\left\langle f,\phi\right\rangle\quad \text{for all }  \phi\in C_{0}^{\infty}(\Omega),\\
\textnormal{(b)}\quad & \mathbb{P}u\rightarrow u_{\Di }\quad \text{nontangentially a.e. on } \partial\Omega.
\end{align*}
Moreover, we have
\[    \norm{u}{D_\alpha^p(\Omega)} \leq C \left( \norm{f}{\Bessel{\alpha-2}{p}(\Omega)} + \norm{u_{\Di}}{\Leb{2}(\partial\Omega)}  \right).\]
In addition, if $u_{\Di} \in \Leb{q}(\Omega)$ and $2\leq q\leq \infty$, then
\[     u \in \Bessel{1/q}{q}(\Omega)+ \Bessel{\alpha}{p}(\Omega).       \]
\end{thm}

\begin{rem*}\leavevmode
\begin{enumerate}[label={\textnormal{(\roman*)}},ref={\roman*},leftmargin=2em]
\item {Suppose that $(\alpha,p)\in \mathscr{A}$ and $\boldb =0$.  Combining Theorems \ref{thm:trace-theorem}, \ref{thm:L2-nontangential}, \ref{thm:Dirichlet-Poisson}, and  \ref{thm:Dirichlet-Poisson-2},} we can prove that for every $f\in \Bessel{\alpha-2}{p}(\Omega)$ and $u_{\Di} \in \Leb{2}(\partial\Omega)$, there exists a unique function $u \in \haSob{1/2}{2}(\Omega) \oplus\oBessel{\alpha}{p}(\Omega) $ satisfying
\begin{align*}
\textnormal{(a)}\quad &\displaystyle -\int_{\Omega}u \triangle\phi \myd{x}=\left\langle f,\phi\right\rangle\quad \text{for all }  \phi\in C_{0}^{\infty}(\Omega),\\
\textnormal{(b)}\quad & \mathbb{P}u\rightarrow u_{\Di }\quad \text{nontangentially a.e. on } \partial\Omega.
\end{align*}
where $\mathbb{P}$ is the projection operator from $\haSob{1/2}{2}(\Omega)\oplus \oBessel{\alpha}{p}(\Omega)$ onto $\haSob{1/2}{2}(\Omega)$. Since $u- \mathbb{P} u \in \oBessel{\alpha}{p}(\Omega)$, it follows from Theorem \ref{thm:trace-theorem} (ii) that  $\Tr (u- \mathbb{P} u )=0$ on $\partial\Omega$.

\item Suppose that $f\in \Bessel{\alpha-2}{p}(\Omega)$ and $u_{\Di}\in \Leb{q}(\partial\Omega)$, where $(\alpha,p)\in \mathscr{A}\cap \mathscr{B}$ and $2\leq q<\infty$ satisfy
\[   \frac{1}{q}\leq \alpha\quad \text{and}\quad  \frac{1}{p}-\frac{\alpha}{n} \leq \frac{1}{q}-\frac{1/q}{n}. \]
Then by Theorem \ref{thm:Sobolev-embedding}, $\Bessel{\alpha}{p}(\Omega)\hookrightarrow \Bessel{1/q}{q}(\Omega)$.  Hence if $u\in D^{p}_{\alpha}(\Omega)$ is a solution of  \eqref{eq:D-1} given by Theorem \ref{thm:L2-boundary}, then $u\in \Bessel{1/q}{q}(\Omega)$. However, the solution $u$ cannot be bounded even when $f=0$ and $u_{\Di}$ is constant. Define
\[    u(x) = -\ln|x|\quad \text{and }\quad \boldb(x) = \frac{x}{|x|^2 \ln |x|}\quad \text{for } x\in B_{1/2},\]
where  $B_{r}$ is the open ball of radius $r$ centered at the origin. Then $u$ satisfies 
\[  -\triangle u +\Div(u\boldb)=0\quad \text{in } B_{1/2}\quad \text{and}\quad u=\ln 2\quad \text{on } \partial B_{1/2}.  \]
Note that $u\in \Bessel{1/q}{q}(B_{1/2})$ for all $2\leq q<\infty$ but $u\notin \Leb{\infty}(B_{1/2})$ while $\boldb \in \Leb{n}(B_{1/2})^n$.
\item
One may ask existence of a solution of the problem \eqref{eq:D-2} with boundary data $v_{\Di }\in\Leb 2(\partial\Omega)$.
Following our previous  strategy, we first find a  function $v_1 \in \haSob{1/2}{2}(\Omega)$ such that $v_{1} \rightarrow v_{\Di}$ nontangentially a.e. on $\partial\Omega$.
The second step is to solve the following problem with zero boundary condition:
\begin{equation}\label{eq:perturbed-L2-data-1}
\left\{\begin{alignedat}{2}
-\triangle v_{2}+\boldb\cdot\nabla v_{2}&=-\boldb\cdot\nabla v_{1}&&\quad \text{in }\Omega,\\
v_{2}&=0 &&\quad \text{on }\partial\Omega.
\end{alignedat}\right.
\end{equation}
However, since   $v_1 \in L_{\alpha}^p (\Omega )$ with $p=2$ and $\alpha =1/2 <1$, Lemma \ref{lem:basic-estimates} (ii) cannot be used to show  that $\boldb\cdot\nabla v_{1} \in L_{\beta-2}^q$ for some $(\beta , q)\in \mathscr{A}\cap\mathscr{B}$.  Hence  existence of a solution $v_2$ of (\ref{eq:perturbed-L2-data-1}) cannot be deduced from Theorem \ref{thm:Dirichlet-problems} (ii). It seems to be still open to prove existence of   solutions of  \eqref{eq:D-2} with boundary data $v_{\Di }\in\Leb 2(\partial\Omega)$ for general $\boldb \in \Leb{n}(\Omega)^n$. It should be remarked that the problem  \eqref{eq:D-2}  can be solved  for any $v_{\Di }\in\Leb 2(\partial\Omega)$ if $\boldb$ is more regular. For instance, Sakellaris \cite{MR4014299} recently proved that if $\boldb \in \Leb{q}(\Omega)^n$, $q>n\geq 3$, and $\Div \boldb \geq 0$ in $\Omega$, then for every $v_{\Di}\in \Leb{2}(\partial\Omega)$, there exists a unique solution $v$ of   \eqref{eq:D-2}  with $g=0$ which satisfies the boundary condition in the sense of nontangential convergence.
\end{enumerate}
\end{rem*}

\subsection{The Neumann problems}
In this subsection, we state the main result for the Neumann problems \eqref{eq:N-1} and \eqref{eq:N-2} on bounded Lipschitz domains in $\mathbb{R}^n$, $n\geq 3$.

\begin{thm} \label{thm:Neumann} Let $n\geq 3$ and  $(\alpha,p)\in \mathscr{A}\cap\mathscr{B}$.
\begin{enumerate}[label={\textnormal{(\roman*)}},ref={\roman*},leftmargin=2em]
\item There exists a positive function $\hat{u}\in \Bessel{\alpha}{p}(\Omega)$  satisfying
\[     \action{\nabla \hat{u},\nabla \phi}  -\action{\hat{u}\boldb,\nabla \phi} = 0 \]
for all $\phi \in \Bessel{2-\alpha}{p'}(\Omega)$. In fact, $\hat{u}\in \Bessel{\beta}{q}(\Omega)$ for all $(\beta,q)\in \mathscr{A}\cap\mathscr{B}$.
\item For every $f\in\oBessel{\alpha-2}p(\Omega)$ and $u_{\Neu}\in\bes{\alpha-1-1/p}p(\partial\Omega)$ satisfying the compatibility condition $\action{f,1}+\action{u_{\Neu},1}=0$,
there exists a unique function  $u\in\Bessel{\alpha}p(\Omega)$
 with  $\int_\Omega  u \myd{x}=0$ such that
 \begin{equation*}
    \action{\nabla u,\nabla \phi} -\action{u\boldb,\nabla \phi} = \action{f,\phi}+\action{u_{\Neu},\Tr \phi}
\end{equation*}
for all $\phi \in \Bessel{2-\alpha}{p'}(\Omega)$.
 Moreover,  $u$ satisfies
\[
\norm u{\Bessel{\alpha}p(\Omega)}\le C\left(\norm f{\oBessel{\alpha-2}p(\Omega)}+\norm{u_{\Neu}}{\bes{\alpha-1-1/p}p(\partial\Omega)}\right)
\]
for some constant $C=C\left(n,\alpha,p,\boldb,\Omega\right)$.
\item For every $g\in\oBessel{-\alpha}{p'}(\Omega)$ and $v_{\Neu}\in\bes{1/p-\alpha}{p'}(\partial\Omega)$ satisfying the compatibility condition $\action{g,\hat{u}} + \action{v_{\Neu},\Tr \hat{u}} =0$,
there exists a unique function $v\in\Bessel{2-\alpha}{p'}(\Omega)$ with $\int_\Omega v \hat{u} \myd{x}=0$ such that
\begin{equation*}
 \action{\nabla v,\nabla \psi} -\action{\boldb\cdot\nabla v,\nabla \psi}  = \action{g,\psi}+\action{v_{\Neu},\Tr \psi}
\end{equation*}
for all $\psi \in \Bessel{\alpha}{p}(\Omega)$.
 Moreover,   $v$ satisfies
\[
\norm v{\Bessel{2-\alpha}{p'}(\Omega)}\le C\left(\norm g{\oBessel{-\alpha}{p'}(\Omega)}+\norm{v_{\Neu}}{\bes{1/p-\alpha}{p'}(\partial\Omega)}\right)
\]
for some constant $C=C\left(n,\alpha,p,\boldb,\Omega\right)$.
\end{enumerate}
\end{thm}
\begin{rem*}\leavevmode
\begin{enumerate}
\item[\rm (i)] When $(\alpha,p)=(1,2)$, Theorem \ref{thm:Neumann} was already shown by Droniou-V\`azquez \cite{MR2476418}.    Recently, $\Bessel{1}{p}$-results were obtained  by  Kang-Kim \cite{MR3623550} for  general elliptic equations of second order:
\begin{equation*}
-\Div(A\nabla u)+\Div(u\boldb) = f,\quad -\Div(A^{t}\nabla v) -\boldb \cdot \nabla v =g\quad \text{in } \Omega,
\end{equation*}
 provided that the matrix  $A$ has a small BMO semi-norm and the boundary $\partial\Omega$ has a small Lipschitz constant.
\item[\rm (ii)] Theorem \ref{thm:Neumann} extends the results in \cite{MR2476418,MR1658089} for  arbitrary Lipschitz domains as well as the results  in \cite{MR3623550} for  Lipschitz domains  having  small Lipschitz constants  when $A$ is the identity matrix.
\end{enumerate}  
\end{rem*}

It was already observed in  Example \ref{example:Neumann-fails}  that the functions $u$ and $v$ of  Theorem \ref{thm:Neumann}  (ii) and (iii) may not solve the Neumann problems \eqref{eq:N-1} and \eqref{eq:N-2}, respectively. However, if $f$ and $g$ are sufficiently regular, then these functions become solutions of the problems \eqref{eq:N-1} and \eqref{eq:N-2}.

\begin{thm}\label{thm:real-Neumann-problem} Let $n\geq 3$ and $(\alpha,p)\in \mathscr{A}\cap\mathscr{B}$.
\begin{enumerate}[label={\textnormal{(\roman*)}},ref={\roman*},leftmargin=2em]
\item Assume that $(\beta,q)$ satisfies
\[   \alpha\leq \beta,\quad  0<\frac{1}{q}<\beta-1 \leq 1,   \quad\text{and}\quad\frac{1}{p}-\frac{\alpha}{n} \geq \frac{1}{q}-\frac{\beta}{n}. \]
 For every $f\in \Bessel{\beta-2}{q}(\Omega)$ and $u_{\Neu} \in \bes{\alpha-1-1/p}{p}(\partial\Omega)$ satisfying the compatibility condition $ \action{f,1}+\action{u_{\Neu},1}=0$, there exists a unique function  $u\in \Bessel{\alpha}{p}(\Omega)$ with $\int_\Omega u  \myd{x}=0$ such that
\begin{equation}
\left\{\begin{alignedat}{2}
-\triangle u+\Div(u\boldb)& =f & \quad & \text{in }\Omega,\\
\gamma_{\nu} (\nabla u-u\boldb)& =u_{\Neu} & \quad & \text{on }\partial\Omega.\\
\end{alignedat}\right.
\end{equation}
\item   Assume that $(\beta,q)$ satisfies
\[   \beta\leq \alpha,\quad 0\leq \beta<\frac{1}{q}<1,\quad \text{and }\quad  \frac{1}{p}-\frac{\alpha}{n} = \frac{1}{q}-\frac{\beta}{n}. \]
 For every $g\in \Bessel{-\beta}{q'}(\Omega)$ and $v_{\Neu} \in \bes{1/p-\alpha}{p'}(\partial\Omega)$ satisfying the compatibility condition $ \action{g,\hat{u}}+\action{v_{\Neu},\Tr \hat{u}}=0$, there exists a unique  function  $v\in \Bessel{2-\alpha}{p'}(\Omega)$ with $\int_\Omega v\hat{u}\myd{x}=0$ such that
\begin{equation}
\left\{\begin{alignedat}{2}
-\triangle v-\boldb \cdot \nabla v& =g & \quad & \text{in }\Omega,\\
\gamma_{\nu} (\nabla v)& =v_{\Neu} & \quad & \text{on }\partial\Omega.
\end{alignedat}\right.
\end{equation}
Here $\hat{u}$ is the function in Theorem \ref{thm:Neumann} (i).
\end{enumerate}
\end{thm}

See Figure \ref{fig:admissible-2} for Theorem \ref{thm:real-Neumann-problem} (i). 

\begin{figure}[h]
\begin{centering}
\includegraphics[width=0.8\textwidth,page=3]{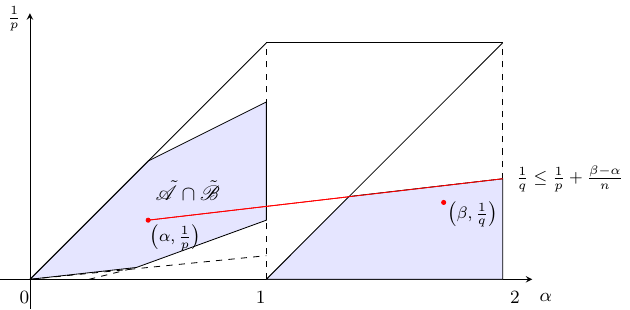}
\par\end{centering}
\caption{\label{fig:admissible-2}}
\end{figure}

\begin{rem*} 
Our proofs of Theorems \ref{thm:Neumann}  and \ref{thm:real-Neumann-problem} cannot be adapted  to obtain corresponding results on a bounded Lipschitz or even $C^1$-domain $\Omega$ in $\mathbb{R}^2$. First, it remains  open to prove unique solvability in $\Bessel{1}{p}(\Omega)$ for the Neumann problems $(N)$ and $(N')$ with $\boldb \in \Leb{2}(\Omega)^2$. Second, our proofs of Theorems \ref{thm:Neumann}  and \ref{thm:real-Neumann-problem} are based on  Theorem \ref{thm:MT-Neumann} which was proved by Mitrea-Taylor \cite{MR1781631} only for $n \ge 3$. 
\end{rem*}

\section{Proofs of Theorems \ref{thm:Dirichlet-problems} and \ref{thm:L2-boundary}}\label{sec:Dirichlet}
The purpose of this section is to prove Theorems \ref{thm:Dirichlet-problems} and \ref{thm:L2-boundary} which are  concerned with unique solvability for the Dirichlet problems  with boundary data in $\bes{\alpha}{p}(\partial\Omega)$ and $\Leb{2}(\partial\Omega)$, respectively.

\subsection{Proof of Theorem \ref{thm:Dirichlet-problems} }

First, we  reduce the problems \eqref{eq:D-1} and \eqref{eq:D-2} to the problems with trivial boundary data.  In the case of the Dirichlet problem \eqref{eq:D-1}, let $(\alpha,p)\in \mathscr{A}\cap\mathscr{B}$ be fixed. Set $h=\mathcal{E}u_{\Di}$, where $\mathcal{E}:\bes{\alpha-1/p}{p}(\partial\Omega)\rightarrow\Bessel{\alpha}p(\Omega)$ is the extension operator given in Theorem \ref{thm:trace-theorem}.  Then
\begin{equation*}
h\in\Bessel{\alpha}p(\Omega)\quad \text{and }\quad \norm h{\Bessel{\alpha}p(\Omega)}\le C\norm{u_{\Di}}{\bes{\alpha-1/p}{p}(\partial\Omega)}
\end{equation*}
for some constant $C=C\left(n,\alpha,p,\Omega\right)$. By Lemma \ref{lem:basic-estimates}, we have
\[\Div\left(h\boldb\right)\in\Bessel{\alpha-2}p(\Omega)\]
and so
\[ \tilde{f}:=f+\triangle h-\Div\left(h\boldb\right)\in\Bessel{\alpha-2}p(\Omega).\]
Note also that
\begin{align*}
\norm{\tilde{f}}{\Bessel{\alpha-2}{p}(\Omega)} &\leq C\left(\norm{f}{\Bessel{\alpha-2}{p}(\Omega)} + \norm{h}{\Bessel{\alpha}{p}(\Omega)}\right) \\
&\leq C\left(\norm{f}{\Bessel{\alpha-2}{p}(\Omega)} + \norm{u_{\Di}}{\bes{\alpha-1/p}{p}(\Omega)}\right)
\end{align*}
for some constant $C=C(n,\alpha,p,\boldb,\Omega)>0$. Hence the problem \eqref{eq:D-1} is reduced to the following   problem:
\begin{equation}\label{eq:homogeneous-problem}
\left\{\begin{alignedat}{2}
-\triangle w+\Div\left(w\boldb\right) & =\tilde{f} & \quad & \text{in }\Omega,\\
w & =0 & \quad & \text{on }\partial\Omega.\nonumber
\end{alignedat}\right.
\end{equation}
One can do a similar reduction to the problem  \eqref{eq:D-2}  into  the problem with trivial boundary data.

Hence from now on, we focus on the solvability for the problems   \eqref{eq:D-1} and \eqref{eq:D-2} with trivial boundary data, that is,  $u_{\Di}=v_{\Di}=0$.

First of all, it follows from Theorems \ref{thm:Dirichlet-Poisson} and \ref{thm:Dirichlet-Poisson-2}  that for each $(\alpha,p)\in \mathscr{A}$,   the operator $\mathcal{L}_{\alpha,p}^{0,D}: \oBessel{\alpha}{p}(\Omega) \rightarrow \Bessel{\alpha-2}{p}(\Omega)$ defined by
\[   \mathcal{L}_{\alpha,p}^{0,D} u = -\triangle u \]is bijective. Let $(\alpha,p)\in \mathscr{A}\cap\mathscr{B}$ be fixed. Recall from  Lemma \ref{lem:compactness-of-operators} that
\[ \mathcal{B}_{\alpha,p}^{D}: \oBessel{\alpha}{p}(\Omega)\rightarrow \Bessel{\alpha-2}{p}(\Omega) \quad\text{and}\quad\mathcal{B}^{*}_{2-\alpha,p'}:\oBessel{2-\alpha}{p'}(\Omega)\rightarrow \Bessel{-\alpha}{p'}(\Omega) \]
 are compact linear operators.  Define
\[   \mathcal{L}_{\alpha,p}^{D} = \mathcal{L}_{\alpha,p}^{0,D}  +\mathcal{B}_{\alpha,p}^{D}  \quad \text{and }\quad  \mathcal{L}_{2-\alpha,p'}^{*,D} = \mathcal{L}_{2-\alpha,p'}^{0,D}  +\mathcal{B}_{2-\alpha,p'}^*. \]
Then $\mathcal{L}_{\alpha,p}^{D}$ is a bounded linear operator from $\oBessel{\alpha}{p}(\Omega)$ to $\Bessel{\alpha-2}{p}(\Omega)$ and $\mathcal{L}_{2-\alpha,p'}^{*,D}$ is a bounded linear operator from $\oBessel{2-\alpha}{p'}(\Omega)$ to $\Bessel{-\alpha}{p'}(\Omega)$.   Since $\mathcal{L}_{\alpha,p}^{0,D}$ and $\mathcal{L}_{2-\alpha,p'}^{0,D}$ are bijective,  we have
\[  \mathcal{L}_{\alpha,p}^{D} = \mathcal{L}_{\alpha,p}^{0,D} \circ \left[I +  \left(\mathcal{L}_{\alpha,p}^{0,D}\right)^{-1} \circ \mathcal{B}_{\alpha,p}^{D}\right] \]
and
\[  \mathcal{L}_{2-\alpha,p'}^{*,D} = \left[I +   \mathcal{B}_{2-\alpha,p'}^* \circ \left( \mathcal{L}_{2-\alpha,p'}^{0,D}\right)^{-1} \right] \circ \mathcal{L}_{2-\alpha,p'}^{0,D}. \]
Here $I$ denotes the identity operator on a Banach space. Hence
\begin{equation}\label{eq:Dirichlet-kernel-1}
\ker\mathcal{L}_{\alpha,p}^D = \ker \left(I +  \left(\mathcal{L}_{\alpha,p}^{0,D}\right)^{-1} \circ \mathcal{B}_{\alpha,p}^{D}\right)
\end{equation}
and
\begin{equation}\label{eq:Dirichlet-kernel-2}
\ker\mathcal{L}_{2-\alpha,p'}^{*,D} = \left(\mathcal{L}_{2-\alpha,p'}^{0,D}\right)^{-1}\left[\ker \left(I +   \mathcal{B}_{2-\alpha,p'}^* \circ \left( \mathcal{L}_{2-\alpha,p'}^{0,D}\right)^{-1}\right)\right].
\end{equation}
On the other hand, since $\mathcal{B}_{\alpha,p}^{D}$ and $\mathcal{B}_{2-\alpha,p'}^{*}$ are compact, it follows from the  Riesz-Schauder theory (see e.g.\ \cite[Theorem 6.6]{MR2759829}) that  the operator
\[  I +  \left(\mathcal{L}_{\alpha,p}^{0,D}\right)^{-1} \circ \mathcal{B}_{\alpha,p}^{D} \]
is injective if and only if it is surjective, and
\begin{equation}\label{eq:Dirichlet-duality}
\dim \ker \left(I +  \left(\mathcal{L}_{\alpha,p}^{0,D}\right)^{-1} \circ \mathcal{B}_{\alpha,p}^{D}\right) = \dim \ker \left(I +  \left[\left(\mathcal{L}_{\alpha,p}^{0,D}\right)^{-1} \circ \mathcal{B}_{\alpha,p}^{D}\right]^{\prime} \right)<\infty,
\end{equation}
where $\left[\left(\mathcal{L}_{\alpha,p}^{0,D}\right)^{-1} \circ \mathcal{B}_{\alpha,p}^{D}\right]^{\prime}$ denotes the adjoint operator of $\left(\mathcal{L}_{\alpha,p}^{0,D}\right)^{-1} \circ \mathcal{B}_{\alpha,p}^{D}$.  By \eqref{eq:duality-B-2}, we  easily get
\begin{equation}\label{eq:adjoint-characterization}
     \left[\left(\mathcal{L}_{\alpha,p}^{0,D}\right)^{-1} \circ \mathcal{B}_{\alpha,p}^{D}\right]^{\prime} = \mathcal{B}_{2-\alpha,p'}^* \circ \left(\mathcal{L}_{2-\alpha,p'}^{0,D}\right)^{-1}.
\end{equation}
Therefore, it follows from \eqref{eq:Dirichlet-kernel-1}, \eqref{eq:Dirichlet-kernel-2}, \eqref{eq:Dirichlet-duality}, and \eqref{eq:adjoint-characterization} that
\begin{equation}\label{eq:kernel-dimension-Dirichlet}
  \dim \ker \mathcal{L}_{\alpha,p}^{D} = \dim \ker \mathcal{L}_{2-\alpha,p'}^{*,D}<\infty.
\end{equation}
We shall show that the kernels of $\mathcal{L}_{\alpha,p}^{D}$ and $\mathcal{L}_{2-\alpha,p'}^{*,D}$ are trivial. We first consider  the special case when $(1,p)\in \mathscr{A}\cap \mathscr{B}$.
\begin{lem}\label{lem:L-1p-result}  Let $(1,p)\in \mathscr{A}\cap \mathscr{B}$. Then
\begin{equation}\label{eq:D-kernel}
\ker \mathcal{L}_{1,p}^{D} = \ker \mathcal{L}_{1,p'}^{*,D} = \{ 0\}. 
\end{equation} 
\end{lem}
\begin{proof} 
It was shown by Trudinger \cite{T73}, Droniou \cite{MR1908676}  for $p=2<n$ and by Kwon \cite[Theorem 4.1]{K21} for  $p<2=n$  that 
\begin{equation}\label{eq:kernel-1}
   \ker \mathcal{L}_{1,p'}^{*,D}=\{0\}.
\end{equation}
For the case  $n=2$, \eqref{eq:D-kernel} follows from \eqref{eq:kernel-dimension-Dirichlet} and \eqref{eq:kernel-1}. If  $n\geq 3$, then \eqref{eq:D-kernel} also follows from  \eqref{eq:kernel-dimension-Dirichlet} and \eqref{eq:kernel-1} since $\Bessel{1}{q}(\Omega)\hookrightarrow \Bessel{1}{2}(\Omega)$ for $q>2$.  
\end{proof}

For general $(\alpha,p)\in \mathscr{A}\cap\mathscr{B}$, we use the following lemma.

\begin{lem}\label{lem:Sobolev-regularity}  Let   $(\alpha,p) , (\beta,q)\in \mathscr{A}\cap\mathscr{B}$ satisfy
\begin{equation}\label{eq:Sobolev-regularity-condition}
\alpha\leq\beta\quad\text{and}\quad\frac{1}{q}-\frac{\beta}{n}= \frac{1}{p}-\frac{\alpha}{n}.
\end{equation}
Then
\[   \ker \mathcal{L}_{\alpha,p}^{D} \subset \ker \mathcal{L}_{\beta,q}^{D}. \]
\end{lem}
\begin{proof}
Let $u \in \ker \mathcal{L}_{\alpha,p}^D$. Then by Lemma \ref{lem:basic-estimates}, $\Div(u\boldb) \in \Bessel{\beta-2}{q}(\Omega)$. So by Theorems \ref{thm:Dirichlet-Poisson} and \ref{thm:Dirichlet-Poisson-2}, there exists a unique $v\in \oBessel{\beta}{q}(\Omega)$ such that $\triangle v =\Div(u\boldb)$ in $\Omega$. Set $w=u-v$. Then $\triangle w =0$ in $\Omega$ and $w\in \Bessel{\alpha}{p}(\Omega)+\Bessel{\beta}{q}(\Omega)$. Since $(\alpha,p)$ and $(\beta,q)$ satisfy \eqref{eq:Sobolev-regularity-condition}, it follows from Theorem \ref{thm:Sobolev-embedding} that    $\Bessel{\alpha}{p}(\Omega)+\Bessel{\beta}{q}(\Omega) = \Bessel{\alpha}{p}(\Omega)$.  Since $(\alpha,p)\in \mathscr{A}$ and $\Tr w=0$ on $\partial\Omega$, it follows from Theorems \ref{thm:Dirichlet-Poisson} and \ref{thm:Dirichlet-Poisson-2} that $w=0$ in $\Omega$ and so $u=v\in \oBessel{\beta}{q}(\Omega)$. Hence $u\in \ker \mathcal{L}_{\beta,q}^{D}$. This completes the proof of Lemma \ref{lem:Sobolev-regularity}.
\end{proof}
\begin{prop}
\label{prop:uniqueness-of-D} Let $\left(\alpha,p\right)\in{\mathscr{A}}\cap{\mathscr{B}}$. Then
\[   \ker \mathcal{L}_{\alpha,p}^{D} =\ker \mathcal{L}_{2-\alpha,p'}^{*,D} = \{0 \}. \]
\end{prop}
\begin{proof}
By Lemma \ref{lem:L-1p-result} and \eqref{eq:kernel-dimension-Dirichlet}, it suffices to show that there exists $q$  such that
$$
\left(1,q\right)\in{\mathscr{A}}\cap{\mathscr{B}} \quad\mbox{and}\quad \ker \mathcal{L}_{\alpha,p}^{D}\subset \ker \mathcal{L}_{1,q}^{D}.
$$ 
To show this, we will find finitely many points $(\beta_1 , q_1 )$, $(\beta_2 , q_2 )$, ..., $(\beta_N , q_N )$ in ${\mathscr{A}}\cap{\mathscr{B}}$ such that  $(\beta_1 , q_1 )= (\alpha , p)$, $\beta_1 \le \beta_2 \cdots \le \beta_N =1$,  and $1/q_{j+1}-\beta_{j+1}/n \le 1/q_j - \beta_j /n$ for $1\le j < N$. In the $\beta q^{-1}$-plane, the point  $(\beta_{j+1}, 1/q_{j})$ will be obtained by moving from $(\beta_j , 1/q_j )$   along the line with slope $1/n$ to the right or   vertically upward. Then by Lemma \ref{lem:Sobolev-regularity}, we may take $q= q_N$.

First, if $n=2,3$ or if $\varepsilon=1$,   we define $q$ by
\[
\frac{1}{q}-\frac{1}{n}=\frac{1}{p}-\frac{\alpha}{n}.
\]
Then it is easy to check that $\left(1,q\right)\in{\mathscr{A}}\cap{\mathscr{B}}$. Thus it follows from Lemma \ref{lem:Sobolev-regularity} that $\ker \mathcal{L}_{\alpha,p}^{D}\subset \ker \mathcal{L}_{1,q}^{D}$.

Suppose next that $n\ge4$ and $0< \varepsilon<1$. Let $\left\{ \alpha_{j}\right\} $ be a sequence defined inductively by
\[
\alpha_{0}=0,\quad\alpha_{1}=\varepsilon,\quad\text{and}\quad\frac{\alpha_{j+1}-\alpha_{j}}{n}+\alpha_{j}=\frac{\alpha_{j+1}-\varepsilon}{3}\,\,\left(j\geq1\right).
\]
Note that $\alpha_{j+1}$$\left(j\geq1\right)$ is the $\beta$-coordinate of the intersection point of the two straight lines
\[
\frac{1}{q}=\frac{\beta-\alpha_{j}}{n}+\alpha_{j}\quad\text{and}\quad\frac{1}{q}=\frac{\beta-\varepsilon}{3}
\]
in the $\beta q^{-1}$-plane. On the other hand, since
\[
\alpha_{j+1}=\frac{3n-3}{n-3}\alpha_{j}+\frac{n\varepsilon}{n-3}\,\,\left(j\geq1\right),
\]
the sequence $\left\{ \alpha_{j}\right\} $ diverges to $\infty$ as $j \to \infty$.
Hence there exists  $k\geq1$ such that $\alpha_{k}<(1+\varepsilon)/2\leq\alpha_{k+1}$.
Let us redefine
\[
\alpha_{k+1}=\frac{1+\varepsilon}{2}\quad\text{and}\quad\alpha_{k+2}=1 ;
\]
see Figure \ref{fig:iteration-lemma} with $k=3$.

\begin{figure}[h]
\begin{centering}
\includegraphics[width=0.75\textwidth]{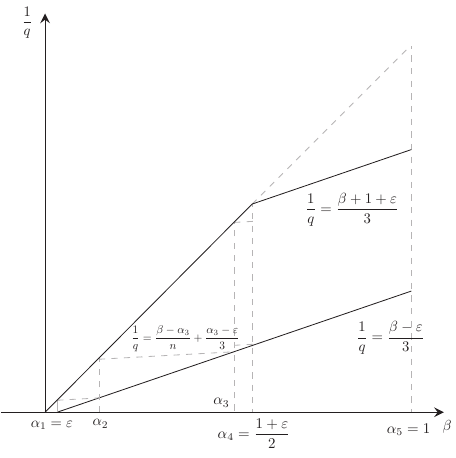}
\par\end{centering}
\caption{\label{fig:iteration-lemma}}
\end{figure}

We claim that if $\alpha_{j-1}\le\alpha<\alpha_{j}$ $\left(1\le j\le k+2\right)$,
then there exists $q$ such that
\[
\left(\alpha_{j},q\right)\in{\mathscr{A}}\cap{\mathscr{B}}\quad\text{and}\quad \ker \mathcal{L}_{\alpha,p}^{D} \subset \ker \mathcal{L}_{{\alpha_{j}},q}^{D}.
\]
If this claim is true, then applying the claim repeatedly, we can show that  there exists
$q$ such that $\left(1,q\right)\in{\mathscr{A}}\cap{\mathscr{B}}$ and $\ker \mathcal{L}_{\alpha,p}^{D} \subset \ker \mathcal{L}_{1,q}^{D}$, which completes the proof.
The proof of the claim consists of two steps.\medskip

\emph{Step 1. }Assume that
\begin{equation}
j=1\quad\text{or}\quad\frac{1}{p}>\frac{\alpha-\alpha_{j}}{n}+\frac{\alpha_{j}-\varepsilon}{3}\,\,\left(2\le j\le k+2\right).\label{eq:step-1-assumption}
\end{equation}
Define $q$ by
\[
\frac{1}{q}-\frac{\alpha_{j}}{n}=\frac{1}{p}-\frac{\alpha}{n}.
\]
Since $\left(\alpha,p\right)\in{\mathscr{A}}\cap{\mathscr{B}}$,
it is easy to check that $\left(\alpha_{j},q\right)\in{\mathscr{A}}\cap{\mathscr{B}}$. Hence  it follows from Lemma \ref{lem:Sobolev-regularity} that $\ker \mathcal{L}_{\alpha,p}^{D}\subset \ker \mathcal{L}_{\alpha_{j},q}^{D}$. \medskip

\emph{Step 2.} Assume that
\[
2\le j\le k+2  \quad\mbox{and}\quad \frac{1}{p}\le\frac{\alpha-\alpha_{j}}{n}+\frac{\alpha_{j}-\varepsilon}{3}.
\]
Let $\beta_{1}$ be the $\beta$-coordinate of the intersection point of the two straight lines
\[
\frac{1}{q}=\frac{\beta-\alpha}{n}+\frac{1}{p}\quad\text{and}\quad\frac{1}{q}=\frac{\beta-\varepsilon}{3}
\]
in the $\beta q^{-1}$-plane; see Figure \ref{fig:iteration-lemma-2}.
\begin{figure}
\begin{centering}
\includegraphics[width=0.75\textwidth]{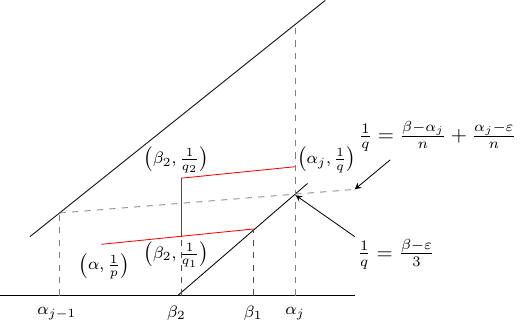}
\par\end{centering}
\caption{\label{fig:iteration-lemma-2}}
\end{figure}
Then $\alpha_{j-1} \le \alpha <\beta_{1}\leq\alpha_{j}$. Define $\beta_{2}$
and $q_{1}$ by
\[
\beta_{2}=\frac{\alpha+\beta_{1}}{2}\quad\text{and}\quad\frac{1}{q_1}-\frac{\beta_{2}}{n}=\frac{1}{p}-\frac{\alpha}{n}.
\]
Then $\alpha<\beta_{2}<\alpha_{j}$ and $\left(\beta_{2},q_{1}\right)\in{\mathscr{A}}\cap{\mathscr{B}}.$
So by Lemma \ref{lem:Sobolev-regularity},  $\ker \mathcal{L}_{\alpha,p}^{D}\subset \ker \mathcal{L}_{\beta_2,q_1}^{D}$.
Choose any $q_{2}$ satisfying
\begin{equation}\label{eq:Step-2-condition}
\frac{\beta_{2}-\alpha_{j}}{n}+\frac{\alpha_{j}-\varepsilon}{3}<\frac{1}{q_{2}}<\min\left\{ \beta_{2},\frac{\beta_{2}+1+\varepsilon}{3}\right\} .
\end{equation}
It is easy to check that $1<q_{2}<q_{1}$ and $\left(\beta_{2},q_{2}\right)\in{\mathscr{A}}\cap{\mathscr{B}}$.
Since $\Bessel{\beta_2}{q_1}(\Omega)\subset \Bessel{\beta_2}{q_2}(\Omega)$, we have $\ker \mathcal{L}_{\beta_2,q_1}^{D}\subset \ker \mathcal{L}_{\beta_2,q_2}^{D}$.  Now, since
$(\beta_2,q_2)$ satisfies \eqref{eq:Step-2-condition}, it follows from  Step 1 that there exists $q$ such that
\[
\left(\alpha_{j},q\right)\in{\mathscr{A}}\cap{\mathscr{B}}\quad\text{and}\quad \ker \mathcal{L}_{\beta_2,q_2}^{D}\subset \ker \mathcal{L}_{\alpha_{j},q}^{D},
\]
which proves the claim.    This completes the proof of Proposition \ref{prop:uniqueness-of-D}.
\end{proof}

By the Riesz-Schauder theory,  Theorem \ref{thm:Dirichlet-problems} is a direct consequence of Proposition \ref{prop:uniqueness-of-D}, but we give a proof for the sake of the completeness.

\begin{proof}[Proof of Theorem \ref{thm:Dirichlet-problems}] It suffices to prove Theorem \ref{thm:Dirichlet-problems} when $u_{\Di}=v_{\Di}=0$. Due to the Riesz-Schauder theory, we already observed that the operator $I +  \left(\mathcal{L}_{\alpha,p}^{0,D}\right)^{-1} \circ \mathcal{B}_{\alpha,p}^{D}$ is injective if and only if it is surjective. By \eqref{eq:Dirichlet-kernel-1} and Proposition \ref{prop:uniqueness-of-D}, we have
\[\ker\left(I +  \left(\mathcal{L}_{\alpha,p}^{0,D}\right)^{-1} \circ \mathcal{B}_{\alpha,p}^{D}\right) =   \ker \mathcal{L}_{\alpha,p}^{D} =  \{0\}.\]
Hence $I +  \left(\mathcal{L}_{\alpha,p}^{0,D}\right)^{-1} \circ \mathcal{B}_{\alpha,p}^{D}$ is bijective. As $\mathcal{L}_{\alpha,p}^{0,D}$ is bijective, we conclude that
\[  \mathcal{L}_{\alpha,p}^{D} = \mathcal{L}_{\alpha,p}^{0,D}+\mathcal{B}_{\alpha,p}^{D} \]
is bijective; that is, given $f\in \Bessel{\alpha-2}{p}(\Omega)$, there exists a unique $u\in \oBessel{\alpha}{p}(\Omega)$ such that $\mathcal{L}_{\alpha,p}^{D}u=f$.  Moreover, it follows from the open mapping theorem that there exists a constant $C=C(n,\alpha,p,\boldb,\Omega)$ such that
\[  \norm{u}{\Bessel{\alpha}{p}(\Omega)} \leq C \norm{f}{\Bessel{\alpha-2}{p}(\Omega)}.  \]
This completes the proof of (i). Following exactly the same argument, we can also prove Theorem \ref{thm:Dirichlet-problems} (ii) of which proof is omitted
\end{proof}

\subsection{Proofs of  Lemma \ref{lem:sumspace-decomposition} and Theorem \ref{thm:L2-boundary}} \label{subsec:L2-boundary}

The existence assertion of Theorem \ref{thm:L2-boundary} was already shown in Section \ref{sec:main-results}. It remains to prove the uniqueness  and regularity assertions of Theorem \ref{thm:L2-boundary}. To do this, we need to prove Lemma \ref{lem:sumspace-decomposition}.  For the case when $\alpha>1/2$, the proofs of Lemma \ref{lem:sumspace-decomposition} and the uniqueness assertion will be based on the following embedding result. 

 \begin{lem}\label{lem:embedding-for-uniqueness}  Let $(\alpha,p)\in \mathscr{A}$. If $1/2 < \alpha\leq  1$, then there is $q$ with $(\alpha,q)\in \mathscr{A}$ such that
\[   \Bessel{1}{\frac{2n}{n+1}}(\Omega)+\Bessel{\alpha}{p}(\Omega)  \hookrightarrow \Bessel{\alpha}{q}(\Omega). \]
In addition, if $(\alpha,p)\in \mathscr{A}\cap\mathscr{B}$, then $q$ can be chosen so that $(\alpha,q)\in \mathscr{A}\cap\mathscr{B}$.
\end{lem}

\begin{proof}
Observe that $\left(1, 2n/(n+1)\right)\in{\mathscr{A}}\cap{\mathscr{B}}$.   By Theorem \ref{thm:Sobolev-embedding},  we have
\[       \Bessel{1}{\frac{2n}{n+1}}(\Omega) + \Bessel{\alpha}{p}(\Omega) \hookrightarrow \Bessel{\alpha}{q_0}(\Omega) +\Bessel{\alpha}{p}(\Omega)  = \Bessel{\alpha}{q}(\Omega),   \]
where
\[    \frac{1}{q_0} -\frac{\alpha}{n} = \frac{n+1}{2n}-\frac{1}{n}\quad \text{and}\quad q=\min(p,q_0). \]

If $p\le q_{0}$, then  $(\alpha,q)=(\alpha,p)\in \mathscr{A}$. Suppose  that $n\geq 3$ and $q=q_{0}<p$. Then since
\[    \alpha>\frac{1}{2},\quad (\alpha,p)\in \mathscr{A},\quad \text{and}\quad \frac{n+1}{2n}<\frac{2+\varepsilon}{3}, \]
we have
\[      \frac{\alpha-\varepsilon}{3}<\frac{1}{p}<  \frac{1}{q} = \frac{\alpha}{n} + \frac{n-1}{2n} <\min\left\{\alpha,\frac{\alpha+1+\varepsilon}{3}\right\}. \]
Hence it follows that  $\left(\alpha,q\right)\in{\mathscr{A}}$. This can be similarly proved for the case $n=2$.  In addition, if  $(\alpha,p)\in \mathscr{B}$,   then  $(\alpha,q)\in \mathscr{B}$ since
\[  \frac{1}{q}-\frac{\alpha}{n} \geq \frac{1}{p}-\frac{\alpha}{n}>0. \]
This completes the proof of Lemma \ref{lem:embedding-for-uniqueness}.
\end{proof}

When $\alpha\leq 1/2$, we will use the following lemma which extends a similar result in  \cite{MR2846167} for the Stokes system in three-dimensional Lipschitz domains.
\begin{lem}
\label{lem:L2-boundary-nontangential-convergence}Let  $(\alpha,p)$ satisfy
\begin{equation}\label{eq:trace-control-2}
1<p<\infty,\quad \frac{1}{p}<\alpha<1+\frac{1}{p} ,\quad \text{and }\quad  \frac{1}{p}-\frac{\alpha}{n}\leq \frac{n-1}{2n}.
\end{equation}
If $u$ is harmonic in $\Omega$ and $u\in\Bessel 1{\frac{2n}{n+1}}(\Omega)+\Bessel{\alpha}p(\Omega)$,
then $u^{*}\in\Leb 2(\partial\Omega)$ and $u\rightarrow\Tr u$
nontangentially a.e. on $\partial\Omega$.
\end{lem}

\begin{proof}
By virtue of an approximation scheme due to Verchota \cite[Theorem 1.12]{MR769382}, 
there are sequences of $C^{\infty}$-domains $\Omega_{j}\subset\Omega$
and homeomorphisms $\Lambda_{j}:\partial\Omega\rightarrow\partial\Omega_{j}$
such that $\Lambda_{j}\left(z\right)\in\gamma\left(z\right)$ for
all $j$ and for all $z\in\partial\Omega$. Also, as $j\rightarrow\infty$,  $\sup_{z\in\partial\Omega}\left|\Lambda_{j}\left(z\right)-z\right|\rightarrow0$
 and $\nu_{j}\left(\Lambda_{j}\left(z\right)\right)\rightarrow \nu\left(z\right)$
for a.e. $z\in\partial\Omega$, where
$\nu_{j}$ is the outward unit normal to $\partial\Omega_{j}$. Moreover,
there exist positive functions $\omega_{j}:\partial\Omega\rightarrow\mathbb{R}$,
bounded away from zero and infinity uniformly in $j$, such that $\omega_{j}\rightarrow1$
a.e. on $\partial\Omega$, and $\int_{E}\omega_{j}d\sigma=\int_{\Lambda_{j}\left(E\right)}d\sigma_{j}$
for any measurable $E\subset\partial\Omega$.

For each $j$, we define $g_{j}=u\,|_{\partial\Omega_{j}}.$  Note that
\begin{equation}\tag{$D_{j}$}
\left\{\begin{alignedat}{2}
\triangle u&=0 &&\quad \text{in }\Omega_{j},\\
u&=g_{j} &&\quad \text{on }\partial\Omega_{j}.
\end{alignedat}\right.
\end{equation}
Hence by a classical result of Verchota \cite[Corollary 3.2]{MR769382}, there exists $f_{j}\in\Leb 2\left(\partial\Omega_{j}\right)$ such that  $u\mid_{\Omega_j}$ can be written as the double layer potential  of $f_{j}$
on $\partial\Omega_{j}$:
\begin{equation}\label{eq:layer-potentials}
u(x)=\int_{\partial\Omega_{j}}\frac{\left\langle x-z_j,\nu_{j}\left(z_j\right)\right\rangle }{\left|x-z_j\right|^{n}}f_{j}\left(z_{j}\right)d\sigma_{j}\left(z_{j}\right),\quad x\in\Omega_{j}.
\end{equation}
and
\begin{equation}\label{eq:compactness-argument}
 \norm{f_{j}}{\Leb 2\left(\partial\Omega_{j}\right)}\le C\norm{g_{j}}{\Leb 2\left(\partial\Omega_{j}\right)}=C\norm u{\Leb 2\left(\partial\Omega_{j}\right)}.
\end{equation}
Here   the constant $C$ in \eqref{eq:compactness-argument} depends only on the Lipschitz character of $\Omega$ {since the Lipschitz charactor of $\Omega_j$ can be uniformly controlled by the Lipschitz character of $\Omega$.} Since $(\alpha,p)$ satisfies \eqref{eq:trace-control-2}, it follows from Corollary \ref{cor:trace-control} and its remark that
\begin{equation}\label{eq:boundary-behavior-trace}
\norm u{\Leb 2\left(\partial\Omega_{j}\right)} \le C\norm u{\Bessel 1{\frac{2n}{n+1}}\left(\Omega_j\right)+\Bessel{\alpha}p\left(\Omega_j\right)} \le C\norm u{\Bessel 1{\frac{2n}{n+1}}(\Omega)+\Bessel{\alpha}p(\Omega)}
\end{equation}
for some constant $C$ depending only on the Lipschitz character of
$\Omega$, $\alpha,p$, and $n$. By the change of variables $z_{j}=\Lambda_{j}\left(z\right)$, we deduce from \eqref{eq:layer-potentials} that
\begin{align}
u\left(x\right) & =\int_{\partial\Omega}\frac{\left\langle x-\Lambda_{j}\left(z\right),\nu_{j}\left(\Lambda_j\left(z\right)\right)\right\rangle }{\left|x-\Lambda_{j}\left(z\right)\right|^{n}}f_{j}\left(\Lambda_{j}\left(z\right)\right)\omega_{j}\left(z\right)d\sigma\left(z\right)\label{eq:convergence-double-layer-potentials}
\end{align}
for all $x \in \Omega_j$. For each $j$, we define $F_{j}\left(z\right)=f_{j}\left(\Lambda_{j}\left(z\right)\right)\omega_{j}\left(z\right)$.
Then by \eqref{eq:compactness-argument} and \eqref{eq:boundary-behavior-trace}, we have
\begin{equation*}
\norm{F_{j}}{\Leb 2(\partial\Omega)} \le C\norm u{\Leb 2\left(\partial\Omega_{j}\right)} \le C\norm u{\Bessel 1{\frac{2n}{n+1}}(\Omega)+\Bessel{\alpha}p(\Omega)}
\end{equation*}
for all $j$. Hence $\left\{ F_{j}\right\} $ is a bounded sequence
in $\Leb 2(\partial\Omega)$, and we may assume
that $F_{j}\rightarrow  F$ weakly  in  $\Leb 2(\partial\Omega)$
for some $F\in\Leb 2(\partial\Omega)$. Thus, by letting
$j\rightarrow\infty$ in \eqref{eq:convergence-double-layer-potentials},
we conclude that $u$ is the double layer potential of $F$ on $\partial\Omega$. 
Therefore, by a consequence of the Coifman-McIntosh-Meyer theorem \cite[Theorem IX]{MR672839}, $u^{*}$ belongs to $\Leb 2(\partial\Omega)$ (see Verchota \cite[Theorem 3.1, Corollary 3.2]{MR769382}). Hence it follows from Proposition \ref{prop:JK-nontangential-maximal-function} that there exists $g\in \Leb{2}(\partial\Omega)$ such that  $u\rightarrow g$ nontangentially a.e. on $\partial\Omega$.

It remains to show that $\Tr u=g$ on $\partial\Omega$. For $0<\delta<1$,
we choose $u_{\delta}\in C^{\infty}\left(\overline{\Omega}\right)$ so
that $\norm{u-u_{\delta}}{\Bessel{\alpha}p(\Omega)+\Bessel 1{\frac{2n}{n+1}}(\Omega)}\le\delta$.
Then  Corollary \ref{cor:trace-control} again gives
\begin{align*}
\norm{u\circ\Lambda_{j}-\Tr u}{\Leb 2(\partial\Omega)} & \le\norm{u\circ\Lambda_{j}-u_{\delta}\circ\Lambda_{j}}{\Leb 2(\partial\Omega)}+\norm{u_{\delta}\circ\Lambda_{j}-\Tr u_{\delta}}{\Leb 2(\partial\Omega)}\\
 & \relphantom{=}+\norm{\Tr u_{\delta}-\Tr u}{\Leb 2(\partial\Omega)}\\
 & \le C\norm{u-u_{\delta}}{\Bessel 1{\frac{2n}{n+1}}(\Omega)+\Bessel{\alpha}p(\Omega)}+\norm{u_{\delta}\circ\Lambda_{j}-u_{\delta}}{\Leb 2(\partial\Omega)}\\
 & \le C\delta+\norm{u_{\delta}\circ\Lambda_{j}-u_{\delta}}{\Leb 2(\partial\Omega)}.
\end{align*}
Since $\sup_{z\in\partial\Omega}\left|\Lambda_{j}\left(z\right)-z\right|\rightarrow0$  as $j\rightarrow\infty$, we have
\[
\limsup_{j\rightarrow\infty}\,\norm{u\circ\Lambda_{j}-\Tr u}{\Leb 2(\partial\Omega)}\le C\delta.
\]
Since $0<\delta<1$ is arbitrary, it follows that $u\circ\Lambda_{j}\rightarrow\Tr u$ in $\Leb 2(\partial\Omega)$ as $j\rightarrow\infty$. On the other hand, since $\Lambda_{j}\left(x\right)\in\gamma\left(x\right)$ for any $x\in\partial\Omega$, we have $\left|u\circ\Lambda_{j}\right|\le u^* \in\Leb 2(\partial\Omega)$. Hence by the dominated convergence theorem, $u\circ\Lambda_{j}\rightarrow g$ in $\Leb 2(\partial\Omega)$ as $j\rightarrow\infty$. This completes the proof of Lemma \ref{lem:L2-boundary-nontangential-convergence}.
\end{proof}

These lemmas lead us to the proofs of Lemma \ref{lem:sumspace-decomposition} and Theorem \ref{thm:L2-boundary}.

\begin{proof}[Proof of Lemma \ref{lem:sumspace-decomposition}]
Suppose first that $\alpha>1/2$. Then by Lemma \ref{lem:embedding-for-uniqueness}, there exists $q$ with $(\alpha,q)\in \mathscr{A}$ such that
\[       \oBessel{\alpha}{p}(\Omega) + \oBessel{1}{\frac{2n}{n+1}}(\Omega) \hookrightarrow \oBessel{\alpha}{q}(\Omega). \]
Since $(\alpha,q) \in \mathscr{A}$, it follows from Theorems \ref{thm:Dirichlet-Poisson} and \ref{thm:Dirichlet-Poisson-2}  that
\[     \oBessel{\alpha}{q}(\Omega) \cap \haSob{1/2}{2}(\Omega) =\{0\}.\]
This implies the assertion when $\alpha > 1/2$. Next, suppose that $0<\alpha \leq 1/2$ and
\[   u\in \haSob{1/2}{2}(\Omega) \cap \left( \oBessel{1}{\frac{2n}{n+1}}(\Omega) + \oBessel{\alpha}{p}(\Omega)\right). \]
Since $u$ is harmonic in $\Omega$ and $(\alpha,p)$ satisfies \eqref{eq:trace-control-2}, it follows from Lemma \ref{lem:L2-boundary-nontangential-convergence} that  $u^* \in \Leb{2}(\partial\Omega)$ and $u \rightarrow \Tr u =0$ nontangentially a.e.\ on $\partial \Omega$. Hence by    Theorem \ref{thm:L2-nontangential},  $u=0$ in  $\Omega$. This completes the proof of Lemma \ref{lem:sumspace-decomposition}.
\end{proof}

\begin{proof}[Proof of Theorem \ref{thm:L2-boundary}]  The existence assertion of Theorem \ref{thm:L2-boundary} was already proved in Section \ref{sec:main-results}. To prove the uniqueness assertion, suppose that $u \in D_{\alpha}^{p}(\Omega)$ satisfies (a) and (b) in Theorem \ref{thm:L2-boundary} with $(f,u_{\Di})=(0,0)$. Then by  Theorem \ref{thm:L2-nontangential},  $\mathbb{P}u =0$ and so $u \in \oBessel{1}{\frac{2n}{n+1}}(\Omega)+\oBessel{\alpha}{p}(\Omega) $.  If  $\alpha>1/2$, then it follows from Lemma \ref{lem:embedding-for-uniqueness} that $u\in \oBessel{\alpha}{q}(\Omega)$ for some $q$ satisfying $(\alpha,q) \in \mathscr{A}\cap\mathscr{B}$. Hence by Theorem \ref{thm:Dirichlet-problems}, $u$ is identically zero in $\Omega$.  Suppose thus that $0<\alpha \leq 1/2$. Since $(\alpha,p)$ satisfies \eqref{eq:trace-control-2}, it follows from  Corollary \ref{cor:trace-control} and Lemma \ref{lem:basic-estimates} that  $u\in \Leb{\frac{2n}{n-1}}(\Omega)$ and   $\Div(u\boldb) \in \Bessel{-1}{\frac{2n}{n+1}}(\Omega)$. Hence by  Theorem \ref{thm:Dirichlet-problems} (i), there exists $v\in \oBessel{1}{\frac{2n}{n+1}}(\Omega)$ such that $\triangle v =\Div(u\boldb)$ in $\Omega$. Define $w=u-v$. Then $\triangle w=0$ in $\Omega$ and $w \in \oBessel{1}{\frac{2n}{n+1}}(\Omega)+\oBessel{\alpha}{p}(\Omega)$. So by Lemma \ref{lem:L2-boundary-nontangential-convergence}, $w^* \in \Leb{2}(\partial\Omega)$ and $w\rightarrow 0$ nontangentially a.e.\ on $\partial\Omega$. Hence it follows from Proposition \ref{prop:JK-nontangential-maximal-function} and Theorem \ref{thm:L2-nontangential} that $w=0$ in $\Omega$ and so $u=v \in \oBessel{1}{\frac{2n}{n+1}}(\Omega)$. Thus, it follows from Theorem \ref{thm:Dirichlet-problems} that $u=0$ in $\Omega$. This completes the proof of the uniqueness assertion of Theorem \ref{thm:L2-boundary}.

To prove the regularity assertion, we write $u=u_1+u_2+u_3$, where $u_1\in \haSob{1/2}{2}(\Omega)$, $u_1\rightarrow u_{\Di}$ nontangentially a.e. on $\partial\Omega$, $u_2\in \oBessel{1}{\frac{2n}{n+1}}(\Omega)$ is a solution of \eqref{eq:perturbed-L2-data}, and $u_3 \in \Bessel{\alpha}{p}(\Omega)$ is a solution of \eqref{eq:D-1}  with $u_{\Di}=0$.    Suppose first that $u_{\Di} \in \Leb{q}(\partial\Omega)$ and  $2\leq q<\infty$. Then by Theorem \ref{thm:L2-nontangential}, $u_1 \in \Bessel{1/q}{q}(\Omega)$. Since $q\geq 2$, there exists $(\gamma,r)\in \mathscr{A}\cap\mathscr{B}$ such that
\[   \frac{1}{q}<\gamma\quad\text{and}\quad \frac{1}{q}-\frac{1/q}{n}=\frac{1}{r}-\frac{\gamma}{n}.   \]
By Lemma \ref{lem:basic-estimates}, $\Div(u_1\boldb) \in \Bessel{\gamma-2}{r}(\Omega)$. Hence it follows from Theorem \ref{thm:Dirichlet-problems} (i) and Theorem \ref{thm:Sobolev-embedding} that $u_2 \in \Bessel{\gamma}{r}(\Omega)\hookrightarrow \Bessel{1/q}{q}(\Omega)$.   This proves that
\[    u\in \Bessel{1/q}{q}(\Omega) + \Bessel{\alpha}{p}(\Omega). \]
  Suppose next that $u_{\Di} \in \Leb{\infty}(\partial\Omega)$. By Theorem \ref{thm:L2-nontangential}, we have $u_1 \in \Leb{\infty}(\Omega)$.  Since $(\alpha,p) \in \mathscr{B}$, we have
\[    \frac{1}{p'} - \frac{1-\alpha}{n} \leq \frac{1}{n'}. \]
So by Theorem \ref{thm:Sobolev-embedding} and  H\"older's  inequality, we have
\[ \Bessel{1-\alpha}{p'}(\Omega)\hookrightarrow \Leb{n'}(\Omega)\quad \text{and}\quad  \Div(u_1\boldb) \in \Bessel{\alpha-2}{p}(\Omega). \] Therefore, it follows from  Theorem \ref{thm:Dirichlet-problems} that $u_2 \in \Bessel{\alpha}{p}(\Omega)$, which implies that
\[   u \in \Leb{\infty}(\Omega) +  \Bessel{\alpha}{p}(\Omega).\]
This completes the proof of Theorem \ref{thm:L2-boundary}.
\end{proof}

\section{Proofs of Theorems \ref{thm:Neumann} and \ref{thm:real-Neumann-problem}} \label{sec:Neumann}
This section is devoted to the proofs of Theorems  \ref{thm:Neumann} and \ref{thm:real-Neumann-problem} which are concerned with the Neumann problems \eqref{eq:N-1} and \eqref{eq:N-2}.  The following theorem is a special case of a result due to Mitrea-Taylor \cite[Theorem 12.1]{MR1781631}.

\begin{thm}\label{thm:MT-Neumann}
Let  $\lambda>0$ and  $(\alpha,p)\in \mathscr{A}$.  For every $f\in \oBessel{\alpha-2}{p}(\Omega)$ and $u_{\Neu} \in \bes{\alpha-1-1/p}{p}(\partial\Omega)$, there exists a unique function $u\in \Bessel{\alpha}{p}(\Omega)$ satisfying
\[    \action{\nabla u,\nabla \phi}+\lambda \action{u,\phi} = \action{f,  \phi} +\action{u_{\Neu},\Tr \phi}    \]
for all $\phi \in \Bessel{2-\alpha}{p'}(\Omega)$.  Moreover, we have
\[   \norm{u}{\Bessel{\alpha}{p}(\Omega)} \leq C \left( \norm{f}{\Bessel{\alpha-2}{p}(\Omega)} + \norm{u_{\Neu}}{\bes{\alpha-1-1/p}{p}(\partial\Omega)}\right) \]
for some constant $C=C(n,\alpha,p,\boldb,\lambda,\Omega)$.
\end{thm}

For $\lambda > 0$ and $(\alpha,p)\in \mathscr{A}$, the mapping
\[         (u,\phi)\mapsto \action{\mathcal{L}_{\alpha,p}^{0,\lambda} u,\phi}=\action{\nabla u,\nabla \phi} +\lambda \action{u,\phi} \]
defines a bounded linear operator $\mathcal{L}_{\alpha,p}^{0,\lambda}$ from $\Bessel{\alpha}{p}(\Omega)$ to $\oBessel{\alpha-2}{p}(\Omega)$. Moreover,  $\mathcal{L}_{\alpha,p}^{0,\lambda}$ is   bijective by Theorem \ref{thm:MT-Neumann}. Note also that
\begin{equation}\label{eq:lambda-duality-Laplacian}
\action{\mathcal{L}_{\alpha,p}^{0,\lambda} u,v} = \action{\mathcal{L}_{2-\alpha,p'}^{0,\lambda} v,u}\quad \text{for all } (u,v) \in \Bessel{\alpha}{p}(\Omega) \times \Bessel{2-\alpha}{p'}(\Omega).
\end{equation}

Now let $(\alpha,p)\in \mathscr{A}\cap\mathscr{B}$ be fixed. By Lemma \ref{lem:basic-estimates}, the mapping
\[    (u,v)\in \Bessel{\alpha}{p}(\Omega)\times \Bessel{2-\alpha}{p'}(\Omega)  \mapsto \action{\mathcal{L}_{\alpha,p}^{N} u,v} =\action{\nabla u,\nabla v} -\action{u\boldb,\nabla v}\]
defines a bounded linear operator $\mathcal{L}_{\alpha,p}^{N}$ from $\Bessel{\alpha}{p}(\Omega)$ to $\oBessel{\alpha-2}{p}(\Omega)$. Also,  the mapping
\[   (u,v)\in \Bessel{\alpha}{p}(\Omega)\times \Bessel{2-\alpha}{p'}(\Omega)\mapsto \action{\mathcal{L}_{2-\alpha,p'}^{*,N} v,u} =\action{\nabla v,\nabla u} -\action{\boldb \cdot \nabla v, u} \]
defines a bounded linear operator $\mathcal{L}_{2-\alpha,p'}^{*,N}$ from $\Bessel{2-\alpha}{p'}(\Omega)$ to $\oBessel{-\alpha}{p'}(\Omega)$. Let $\mathcal{I}_p : \Leb{p}(\Omega) \rightarrow (\Leb{p'}(\Omega))'$ be the isomorphism defined by
\[   \action{\mathcal{I}_{p} u,v} = \int_\Omega uv \myd{x}\quad \text{for all } (u,v) \in \Leb{p}(\Omega)\times \Leb{p'}(\Omega).  \]

Suppose that $\lambda>0$. Recall from Lemma \ref{lem:compactness-of-operators} that
\[ \mathcal{B}_{\alpha,p}^{N}: \Bessel{\alpha}{p}(\Omega)\rightarrow \Bessel{\alpha-2}{p}(\Omega) \quad\text{and}\quad\mathcal{B}^{*}_{2-\alpha,p'}:\Bessel{2-\alpha}{p'}(\Omega)\rightarrow \Bessel{-\alpha}{p'}(\Omega) \]
 are compact linear operators.    By the definitions, we have
\[  \mathcal{L}_{\alpha,p}^{N} + \lambda \mathcal{I}_p = \mathcal{L}_{\alpha,p}^{0,\lambda} + \mathcal{B}_{\alpha,p}^{N}  \]
and
\[  \mathcal{L}_{2-\alpha,p'}^{*,N} + \lambda \mathcal{I}_{p'} = \mathcal{L}_{2-\alpha,p'}^{0,\lambda} + \mathcal{B}_{2-\alpha,p'}^{*}.  \]
By \eqref{eq:duality-B} and \eqref{eq:lambda-duality-Laplacian}, we have
\begin{equation}\label{eq:lambda-duality-Neumann}
\action{(\mathcal{L}_{\alpha,p}^{N}+\lambda\mathcal{I}_{p})u,v} = \action{(\mathcal{L}_{2-\alpha,p'}^{*,N}+\lambda\mathcal{I}_{p'})v,u}
\end{equation}
for all $u\in \Bessel{\alpha}{p}(\Omega)$ and $v\in \Bessel{2-\alpha}{p'}(\Omega)$. Since $\mathcal{L}_{\alpha,p}^{0,\lambda}$ and $\mathcal{L}_{2-\alpha,p'}^{0,\lambda}$ are bijective, it follows that
\[     \mathcal{L}_{\alpha,p}^{0,\lambda} +\mathcal{B}_{\alpha,p}^{N} = \mathcal{L}_{\alpha,p}^{0,\lambda} \circ \left[I + \left(\mathcal{L}_{\alpha,p}^{0,\lambda}\right)^{-1}\circ \mathcal{B}_{\alpha,p}^{N} \right] \]
and
\[     \mathcal{L}_{2-\alpha,p'}^{0,\lambda} + \mathcal{B}_{2-\alpha,p'}^* =  \left[I + \mathcal{B}_{2-\alpha,p'}^* \circ {\left( \mathcal{L}_{2-\alpha,p'}^{0,\lambda}\right)}^{-1} \right] \circ \mathcal{L}_{2-\alpha,p'}^{0,\lambda}.  \]
Thus,
\begin{equation}\label{eq:kernel-characterization-Neumann-1}
   \ker(\mathcal{L}_{\alpha,p}^N +\lambda \mathcal{I}_{p}) = \ker\left(I+{\left( \mathcal{L}_{\alpha,p}^{0,\lambda}\right)}^{-1} \circ \mathcal{B}_{\alpha,p}^N\right)
\end{equation}
and
\begin{equation}\label{eq:kernel-characterization-Neumann-2}
  \ker(\mathcal{L}_{2-\alpha,p'}^{*,N} +\lambda \mathcal{I}_{p'})=\mathcal{L}_{2-\alpha,p'}^{0,\lambda}\left[\ker \left(I + \mathcal{B}_{2-\alpha,p'}^* \circ {\left( \mathcal{L}_{2-\alpha,p'}^{0,\lambda}\right)}^{-1} \right)\right].
\end{equation}
On the other hand, since $\mathcal{B}_{\alpha,p}^{N}$ and $\mathcal{B}_{2-\alpha,p'}^{*}$ are compact linear operators, it follows from the Riesz-Schauder theory that $I+(\mathcal{L}_{\alpha,p}^{0,\lambda})^{-1} \circ \mathcal{B}_{\alpha,p}^{N}$ is injective if and only if it is bijective, and
\begin{equation}\label{eq:kernel-duality-1}
 \dim \ker\left(I+{\left( \mathcal{L}_{\alpha,p}^{0,\lambda}\right)}^{-1} \circ \mathcal{B}_{\alpha,p}^N\right) = \dim \ker \left(I + \left[{\left( \mathcal{L}_{\alpha,p}^{0,\lambda}\right)}^{-1} \circ \mathcal{B}_{\alpha,p}^N\right]^{\prime}\right)<\infty.
\end{equation}
 From \eqref{eq:duality-B} and \eqref{eq:lambda-duality-Laplacian}, we easily get
\begin{equation}\label{eq:kernel-adjoint-Neumann}
\left[ \left(\mathcal{L}_{\alpha,p}^{0,\lambda}\right)^{-1} \circ \mathcal{B}_{\alpha,p}^N \right]^{\prime} = \mathcal{B}_{2-\alpha,p'}^* \circ \left( \mathcal{L}_{2-\alpha,p'}^{0,\lambda} \right)^{-1}.
\end{equation}
Thus by \eqref{eq:kernel-characterization-Neumann-1},  \eqref{eq:kernel-characterization-Neumann-2},  \eqref{eq:kernel-duality-1}, and \eqref{eq:kernel-adjoint-Neumann},  we have
\begin{equation}\label{eq:kernel-duality-2}
   \dim \ker(\mathcal{L}_{\alpha,p}^{N} +\lambda \mathcal{I}_{p})  = \dim\ker(\mathcal{L}_{2-\alpha,p'}^{*,N} +\lambda \mathcal{I}_{p'}) <\infty.
\end{equation}

The kernels $\mathcal{L}_{1,2}^{N}+\lambda \mathcal{I}_{2}$ and $\mathcal{L}_{1,2}^{*,N}+\lambda \mathcal{I}_{2}$ have been characterized by  Droniou-V\'azquez \cite[Propositions 2.2 and 5.1]{MR2476418} and Kang-Kim \cite[Lemma 4.5]{MR3623550}.
\begin{lem}\label{lem:L2-results-Neumann-kernel}\leavevmode
\begin{enumerate}[label={\textnormal{(\roman*)}},ref={\roman*},leftmargin=2em]
\item $\ker \mathcal{L}_{1,2}^{N} =\mathbb{R}\hat{u}$\, and\, $\ker\mathcal{L}_{1,2}^{*,N} = \mathbb{R}$\, for some function $\hat{u}\in \Bessel{1}{2}(\Omega)$ satisfying $\hat{u}>0$ a.e. on $\Omega$.
\item $\ker (\mathcal{L}_{1,2}^{N}+\lambda \mathcal{I}_{2}) = \ker (\mathcal{L}_{1,2}^{*,N}+\lambda \mathcal{I}_{2})=\{0\}$\, for all $\lambda>0$.
\end{enumerate}
\end{lem}

The following lemma will be used to characterize the kernels $\mathcal{L}_{\alpha,p}^{N}+\lambda \mathcal{I}_p$ and $\mathcal{L}_{2-\alpha,p'}^{*,N}+\lambda \mathcal{I}_{p'}$ for  $(\alpha,p)\in \mathscr{A}\cap\mathscr{B}$ and $\lambda \geq 0$.
\begin{lem}\label{lem:Neumann-regularity-estimates}
Let $(\alpha,p)\in \mathscr{A}\cap\mathscr{B}$. Then there exists $(1,q)\in \mathscr{A}\cap\mathscr{B}$ such that
\[   \ker(\mathcal{L}_{\alpha,p}^{N}+\lambda \mathcal{I}_{p})\subset \ker (\mathcal{L}_{1,q}^{N} + \lambda \mathcal{I}_{q})\quad \text{for all } \lambda \geq 0.   \]
\end{lem}
\begin{proof}
We first claim that  if $(\alpha,p), (\beta,q)\in \mathscr{A}\cap\mathscr{B}$ satisfy
\[   \alpha \leq \beta\quad \text{and }\quad \frac{1}{q}-\frac{\beta}{n}=\frac{1}{p}-\frac{\alpha}{n},\]
then
\[   \ker \left(\mathcal{L}_{\alpha,p}^{N}+\lambda \mathcal{I}_{p}\right) \subset \ker (\mathcal{L}_{\beta,q}^{N} + \lambda \mathcal{I}_{q})\quad \text{for all } \lambda \geq 0. \]
Let $u\in \ker (\mathcal{L}_{\alpha,p}^N+\lambda \mathcal{I}_{p})$. Then for all $\phi \in C^\infty(\overline{\Omega})$, we have
\begin{equation}\label{eq:Neumann-regularity-1}
\action{\nabla u,\nabla \phi} + \lambda \action{u,\phi} = \action{u\boldb,\nabla \phi} .
\end{equation}
By Lemma \ref{lem:basic-estimates}, the linear functional $\ell$ defined by
\begin{equation*}
\action{\ell,\phi} =\action{u\boldb,\nabla \phi} \quad \text{for all } \phi \in \Bessel{2-\beta}{q'}(\Omega)
\end{equation*}
is bounded on $\Bessel{2-\beta}{q'}(\Omega)$ and  satisfies $\action{\ell,1}=0$. Hence by Theorem \ref{thm:MT-Neumann} and Theorem \ref{thm:Neumann-Poisson}, there exists  $v\in \Bessel{\beta}{q}(\Omega)$ such that
\begin{equation}\label{eq:Neumann-regularity-2}
 \action{\nabla v,\nabla \phi}+\lambda \action{v,\phi} = \action{u\boldb,\nabla \phi}
 \end{equation}
 for all $\phi \in C^\infty(\overline{\Omega})$. Set $w=u-v$. It follows from \eqref{eq:Neumann-regularity-1}, \eqref{eq:Neumann-regularity-2}, and Theorem \ref{thm:Sobolev-embedding} that
\[  w\in \Bessel{\alpha}{p}(\Omega) + \Bessel{\beta}{q}(\Omega) = \Bessel{\alpha}{p}(\Omega) \]
and
\begin{equation}\label{eq:Neumann-regularity-3}
 \action{\nabla w,\nabla \phi}+\lambda \action{w,\phi} =0
\end{equation}
for all  $\phi \in C^\infty(\overline{\Omega})$.
A standard density argument shows that \eqref{eq:Neumann-regularity-3} holds for all $\phi \in \Bessel{2-\alpha}{p'}(\Omega)$. Hence by Theorem  \ref{thm:MT-Neumann} and Theorem \ref{thm:Neumann-Poisson}, $w=c$ for some constant $c$, which implies that $u\in \Bessel{\beta}{q}(\Omega)$. This proves the desired claim.

Now,  following exactly the same argument as in the proof of  Proposition \ref{prop:uniqueness-of-D} except using the claim instead of Lemma \ref{lem:Sobolev-regularity}, we can  prove   Lemma \ref{lem:Neumann-regularity-estimates} of which proof is omitted.
\end{proof}

\begin{prop}\label{prop:Neumann-kernel-characterization}
Let $(\alpha,p)\in \mathscr{A}\cap\mathscr{B}$.
\begin{enumerate}[label={\textnormal{(\roman*)}},ref={\roman*},leftmargin=2em]
\item $\ker \mathcal{L}_{\alpha,p}^{N} = \mathbb{R}\hat{u}$ and $\ker \mathcal{L}_{2-\alpha,p'}^{*,N} =\mathbb{R}$, where $\hat{u}$ is the function in  Lemma \ref{lem:L2-results-Neumann-kernel}.
\item $\ker(\mathcal{L}_{\alpha,p}^{N}+\lambda \mathcal{I}_{p}) = \ker(\mathcal{L}_{2-\alpha,p'}^{*,N}+\lambda \mathcal{I}_{p}) =\{0\}$ for all $\lambda >0$.
\end{enumerate}
\end{prop}
\begin{proof}
(ii)  Fix $\lambda>0$. By Lemma \ref{lem:Neumann-regularity-estimates} and \eqref{eq:kernel-duality-2}, it suffices to show the assertion when $\alpha=1$. Suppose first that $p\geq  2$. Since $\Bessel{1}{p}(\Omega)\subset \Bessel{1}{2}(\Omega)$, we have
\[   \ker(\mathcal{L}_{1,p}^{N}+\lambda\mathcal{I}_{p}) \subset \ker(\mathcal{L}_{1,2}^{N}+\lambda \mathcal{I}_2).\]
Hence it follows from Lemma \ref{lem:L2-results-Neumann-kernel} (ii) and \eqref{eq:kernel-duality-2} that
\begin{equation}\label{eq:trivial-kernel-Neumann-lambda}
\ker(\mathcal{L}_{1,p}^{N}+\lambda\mathcal{I}_{p})=\ker(\mathcal{L}_{1,p'}^{*,N}+\lambda \mathcal{I}_{p'}) =\{0\}.
\end{equation}
   Suppose next that $p<2$. Then since  $\ker(\mathcal{L}_{1,p'}^{*,N}+\lambda\mathcal{I}_{p'}) \subset \ker(\mathcal{L}_{1,2}^{*,N}+\lambda \mathcal{I}_2)$, \eqref{eq:trivial-kernel-Neumann-lambda} also follows from Lemma \ref{lem:L2-results-Neumann-kernel} (ii) and \eqref{eq:kernel-duality-2}.  This completes the proof of (ii).

(i) By the Riesz-Schauder theory, it immediately follows  from (ii) that the operators  $\mathcal{L}_{\alpha,p}^{N} + \mathcal{I}_{p}$ and $\mathcal{L}_{2-\alpha,p'}^{*,N}+\mathcal{I}_{p'}$  are invertible. Then by Theorem \ref{thm:Sobolev-embedding},  the linear operators
\[     \mathcal{K}_{\alpha,p} = \left( \mathcal{L}_{\alpha,p}^{N}+\mathcal{I}_{p}\right)^{-1} \circ \mathcal{I}_{p} : \Leb{p}(\Omega) \rightarrow \Leb{p}(\Omega) \]
and
\[     \mathcal{K}_{2-\alpha,p'}^{*} = \left( \mathcal{L}_{2-\alpha,p'}^{*,N}+\mathcal{I}_{p'}\right)^{-1} \circ \mathcal{I}_{p'} : \Leb{p'}(\Omega) \rightarrow \Leb{p'}(\Omega) \]
are bounded and compact. Since
\[     \int_\Omega (\mathcal{K}_{\alpha,p}u) v \myd{x} = \int_\Omega (\mathcal{K}^{*}_{2-\alpha,p'} v)u \myd{x} \]
for all $(u,v)\in \Leb{p}(\Omega)\times \Leb{p'}(\Omega)$, it follows that
\[   \mathcal{K}_{2-\alpha,p'}^{*} = \mathcal{I}_{p'}^{-1} \circ \mathcal{K}_{\alpha,p}^{\prime} \circ \mathcal{I}_{p'},\]
where $\mathcal{K}^{\prime}_{\alpha,p} : (\Leb{p}(\Omega))' \rightarrow (\Leb{p}(\Omega))'$ is the adjoint operator of $\mathcal{K}_{\alpha,p}$. Note also that
\[     \ker (I -\mathcal{K}_{\alpha,p}) = \ker \mathcal{L}_{\alpha,p}^{N}\quad \text{and}\quad \ker(I-\mathcal{K}_{2-\alpha,p'}^*)=\ker\mathcal{L}_{2-\alpha,p'}^{*,N}. \]
Hence  by the Riesz-Schauder theory, we deduce that
\begin{equation}\label{eq:Neumann-1-image}
  \mathrm{Im} (I-\mathcal{K}_{\alpha,p}) = \left\{  u\in \Leb{p}(\Omega) : \int_\Omega uv \myd{x}=0\quad \text{for all } v\in \ker \mathcal{L}_{2-\alpha,p'}^{*,N}   \right\},
\end{equation}
\begin{equation}\label{eq:Neumann-2-image}
  \mathrm{Im} (I-\mathcal{K}_{2-\alpha,p'}^*) = \left\{  v\in \Leb{p'}(\Omega) : \int_\Omega uv \myd{x}=0\quad \text{for all } u\in \ker \mathcal{L}_{\alpha,p}^{N}   \right\},
\end{equation}
and
\begin{equation}\label{eq:kernel-dimension}
\dim \ker \mathcal{L}_{\alpha,p}^N = \dim\ker \mathcal{L}_{2-\alpha,p'}^{*,N} <\infty.
\end{equation}
But since $\mathbb{R}\subset \ker \mathcal{L}_{2-\alpha,p'}^{*,N}$, we have
\begin{equation}\label{eq:inclusion-dimension-Neumann}
1\leq \dim \ker \mathcal{L}_{2-\alpha,p'}^{*,N}=\dim\ker \mathcal{L}_{\alpha,p}^{N}.
\end{equation}
Hence by Lemma \ref{lem:Neumann-regularity-estimates},  it suffices to show that
\begin{equation}\label{eq:Bessel-1p-characterization}
 \ker \mathcal{L}_{1,q}^{N}  = \mathbb{R}\hat{u}\quad \text{for all } (1,q)\in \mathscr{A}\cap\mathscr{B},
 \end{equation}
where $\hat{u}$ is the same function in Lemma \ref{lem:L2-results-Neumann-kernel}.

Fix $(1,q)\in \mathscr{A}\cap\mathscr{B}$ and suppose first that $q\geq 2$. Since $\Bessel{1}{q}(\Omega)\subset \Bessel{1}{2}(\Omega)$, we have $\ker \mathcal{L}_{1,q}^{N}\subset \ker \mathcal{L}_{1,2}^{N}$. Thus, \eqref{eq:Bessel-1p-characterization} follows from Lemma \ref{lem:L2-results-Neumann-kernel} and \eqref{eq:inclusion-dimension-Neumann}. Suppose next that $q\leq 2$. Then since $\ker \mathcal{L}_{1,q'}^{*,N}\subset \ker \mathcal{L}_{1,2}^{*,N}$ and $\ker \mathcal{L}_{1,2}^{N}\subset \ker \mathcal{L}_{1,q}^{N}$, \eqref{eq:Bessel-1p-characterization} follows from Lemma \ref{lem:L2-results-Neumann-kernel} and  \eqref{eq:inclusion-dimension-Neumann} again.  This completes the proof of Proposition \ref{prop:Neumann-kernel-characterization}.
\end{proof}

Now we are prepared to prove Theorem \ref{thm:Neumann} using Proposition \ref{prop:Neumann-kernel-characterization}.

\begin{proof}[Proof of Theorem  \ref{thm:Neumann}]
By Lemma \ref{lem:L2-results-Neumann-kernel}, there exists a function $\hat{u}>0$ a.e. on $\Omega$ such that  $\ker \mathcal{L}_{1,2}^{N}=\mathbb{R}\hat{u}$. Moreover, by Proposition \ref{prop:Neumann-kernel-characterization} (i), $\hat{u}\in \Bessel{\beta}{q}(\Omega)$  and $\ker \mathcal{L}_{\beta,q}^{N}=\mathbb{R}\hat{u}$ for all $(\beta,q)\in \mathscr{A}\cap\mathscr{B}$. This proves Theorem  \ref{thm:Neumann} (i).

Let us prove  (ii).  Suppose that $f\in\oBessel{\alpha-2}p(\Omega)$ and $u_{\Neu}\in\bes{\alpha-1-1/p}p(\partial\Omega)$ satisfy $\action{f,1}+\action{u_{\Neu},1}=0$.
Define a linear functional $\ell$ by
\[
\action{\ell,\phi}=\action{f,\phi}+\action{u_{\Neu},\Tr\phi}\quad \text{for all } \phi \in \Bessel{2-\alpha}{p'}(\Omega).
\]
Then $\ell\in\oBessel{\alpha-2}p(\Omega)$ and $\action{\ell,1}=0$. By Proposition \ref{prop:Neumann-kernel-characterization}, $\mathcal{L}_{\alpha,p}^{N}+\mathcal{I}_{p}$ is invertible. Hence  there exists a unique $w\in\Bessel{\alpha}p(\Omega)$ such that
\[
\left(\mathcal{L}_{\alpha,p}^{N}+\mathcal{I}_{p}\right)w=\ell.
\]
Note that
\[   \action{\mathcal{I}_{p}w,1} =-\action{\mathcal{L}_{\alpha,p}^{N}w,1} + \action{\ell,1} = 0. \]
Hence by Proposition \ref{prop:Neumann-kernel-characterization} (i) and \eqref{eq:Neumann-1-image}, there exists $u_{0}\in\Bessel{\alpha}p(\Omega)$
such that
\[
\left(I-\mathcal{K}_{\alpha,p}\right)u_{0}=w.
\]
By the definitions of $\mathcal{K}_{\alpha,p}$ and $w$, we have
\[
\mathcal{L}_{\alpha,p}^{N}u_{0}=\left(\mathcal{L}_{\alpha,p}^{N}+\mathcal{I}_{p}\right)w=\ell,
\]
which implies that $u_{0}\in \Bessel{\alpha}{p}(\Omega)$ satisfies
\begin{equation}\label{eq:N-1-integral-equation}
\action{\nabla u_0,\nabla \phi}-\action{u_0 \boldb,\nabla \phi} =\action{f,\phi} + \action{u_{\Neu},\Tr \phi} \quad \text{for all } \phi \in \Bessel{2-\alpha}{p'}(\Omega).
\end{equation}
Hence defining
\[
u=u_{0}-\left( \frac{\int_\Omega u_0 \, dx }{\int_\Omega \hat{u} \, dx } \right)  \hat{u},
\]
we prove the existence assertion of Theorem  \ref{thm:Neumann} (ii).
The $\Bessel{\alpha}p$-estimate for the function $u$ is easily deduced from the boundness of the operators $\left(\mathcal{L}_{\alpha,p}^{N}+\mathcal{I}_{p}\right)^{-1}$ and $\mathcal{K}_{\alpha,p}$. Finally, to prove the uniqueness part,  let $\overline{u}\in\Bessel{\alpha}p(\Omega)$ be another function satisfying \eqref{eq:N-1-integral-equation} and $\int_{\Omega}\overline{u}\myd{x}=0$. Since $u-\overline{u}\in\ker\mathcal{L}_{\alpha,p}^{N}$, it follows from Proposition \ref{prop:Neumann-kernel-characterization} (i) that $u-\overline{u}=c\hat{u}$ for some constant $c$. But $c$ must be zero because
\[
c\int_{\Omega}\hat{u}\myd{x}=\int_{\Omega}u\myd{x}-\int_{\Omega}\overline{u}\myd{x}=0.
\]
This completes the proof of Theorem \ref{thm:Neumann} (ii).  Following exactly the same argument except for
using \eqref{eq:Neumann-2-image} instead of \eqref{eq:Neumann-1-image},
we can also prove Theorem \ref{thm:Neumann} (iii) of which the proof is omitted.
\end{proof}

Using Proposition \ref{prop:normal-trace} and Theorem \ref{thm:Neumann},  we prove Theorem \ref{thm:real-Neumann-problem}.

\begin{proof}[Proof of Theorem  \ref{thm:real-Neumann-problem}]
 The proof of (i) is similar to that of Theorem \ref{thm:real-Neumann-easy-one}. We  only prove  (ii).  Following the argument in the proof of Theorem \ref{thm:real-Neumann-easy-one}, we see that $\Bessel{-\beta}{q'}(\Omega)\hookrightarrow \oBessel{-\alpha}{p'}(\Omega)$. Hence it follows from Theorem \ref{thm:Neumann} (iii) that there exists a unique function $v\in \Bessel{2-\alpha}{p'}(\Omega)$ with $\int_\Omega v\hat{u} \myd{x}=0$ such that
\[
\action{\nabla v,\nabla \psi} -\action{\boldb \cdot \nabla v,  \psi} =\action{g,\psi} +\action{v_{\Neu},\Tr \psi}\quad \text{for all } \psi \in \Bessel{\alpha}{p}(\Omega).
\]
By Lemma \ref{lem:basic-estimates} (ii), we have
\[ -\triangle v =g+\boldb \cdot \nabla v \in \Bessel{-\beta}{q'}(\Omega).\]
Hence it follows from Proposition \ref{prop:normal-trace} that there exists a unique $\gamma_{\nu}(\nabla v) \in \bes{1/p-\alpha}{p'}(\partial\Omega)$ satisfying
\[   \action{\gamma_{\nu}(\nabla v),\Tr \psi} = -\action{g+\boldb \cdot \nabla v, \psi} + \action{\nabla v, \nabla \psi}=\action{v_{\Neu},\Tr \psi}\]
for all $\psi \in \Bessel{\alpha}{p}(\Omega)$. Since $\Tr : \Bessel{\alpha}{p}(\Omega) \rightarrow \bes{\alpha-1/p}{p}(\partial\Omega)$ is surjective, we get $\gamma_{\nu}(\nabla v) =v_{\Neu}$. To show the uniqueness part, let $v \in \Bessel{2-\alpha}{p'}(\Omega)$ be a solution of
\begin{equation*}
\left\{\begin{alignedat}{2}
-\triangle v-\boldb\cdot \nabla v & =0 & \quad & \text{in }\Omega,\\
\gamma_{\nu}(\nabla  v) & =0 & \quad & \text{on }\partial\Omega
\end{alignedat}\right.
\end{equation*}
satisfing  $\int_\Omega v \hat{u} \myd{x}=0$. Then since $\gamma_{\nu}(\nabla v)=0$ and $-\triangle v -\boldb \cdot \nabla v=0$ in $\Omega$, it follows that
\[   0= \action{\gamma_{\nu}(\nabla v),\Tr \psi} = \action{\nabla v,\nabla \psi} + \action{\triangle v,\psi} = \action{\nabla v,\nabla \psi} -\action{\boldb \cdot \nabla v, \psi}\]
for all $\psi \in \Bessel{\alpha}{p}(\Omega)$.  This implies that  $v\in \ker \mathcal{L}_{2-\alpha,p'}^{*,N}$. So it follows from  Proposition  \ref{prop:Neumann-kernel-characterization} that $v=c$ for some constant $c$.  Since $\int_\Omega v \hat{u} \myd{x}=0$ and $\hat{u}>0$ a.e. on $\Omega$, we conclude that $v=c=0$. This  completes the proof of Theorem \ref{thm:real-Neumann-problem}.
\end{proof}

Using the Riesz-Schauder theory with Proposition \ref{prop:Neumann-kernel-characterization} (ii), we can also prove the following theorem whose proof is omitted.
\begin{thm}
\label{thm:Neumann-lambda} Let  $\lambda>0$ and $(\alpha,p)\in \mathscr{A}\cap\mathscr{B}$.
\begin{enumerate}[label={\textnormal{(\roman*)}},ref={\roman*},leftmargin=2em]
\item For every $f\in\oBessel{\alpha-2}p(\Omega)$ and $u_{\Neu}\in\bes{\alpha-1-1/p}p(\partial\Omega)$,
there exists a unique function $u\in\Bessel{\alpha}p(\Omega)$ satisfying
\[    \action{\nabla u,\nabla \phi}-\action{u\boldb,\nabla \phi} + \lambda \action{u,\phi} = \action{f, \phi} +\action{u_{\Neu},\Tr \phi}    \]
for all $\phi \in \Bessel{2-\alpha}{p'}(\Omega)$.  Moreover,  $u$ satisfies
\[
\norm u{\Bessel{\alpha}p(\Omega)}\le C\left(\norm f{\oBessel{\alpha-2}p(\Omega)}+\norm{u_{\Neu}}{\bes{\alpha-1-1/p}p(\partial\Omega)}\right)
\]
for some constant $C=C(n,\alpha,p,\boldb,\lambda,\Omega)$.
\item For every $g\in\oBessel{-\alpha}{p'}(\Omega)$ and $v_{\Neu}\in\bes{1/p-\alpha}{p'}(\partial\Omega)$,
there exists a unique function $v\in\Bessel{2-\alpha}{p'}(\Omega)$ satisfying
\[    \action{\nabla v,\nabla \psi}-\action{\boldb\cdot\nabla v,\psi} + \lambda \action{v,\psi} = \action{g, \psi} +\action{v_{\Neu},\Tr \psi}   \]
for all $\psi \in \Bessel{\alpha}{p}(\Omega)$.  Moreover,   $v$ satisfies
\[
\norm v{\Bessel{2-\alpha}{p'}(\Omega)}\le C\left(\norm{g}{\oBessel{-\alpha}{p'}(\Omega)}+\norm{v_{\Neu}}{\bes{1/p-\alpha}{p'}(\partial\Omega)}\right)
\]
for some constant $C=C(n,\alpha,p,\boldb,\lambda,\Omega)$.
\end{enumerate}
\end{thm}

Theorem \ref{thm:Neumann-lambda} is an extension of Theorem \ref{thm:MT-Neumann}  to more general equations with singular drifts $\boldb$ in $ \Leb{n}(\Omega)^n$. Moreover, following the proof of Theorem \ref{thm:real-Neumann-problem} but using Theorem \ref{thm:Neumann-lambda} instead of Theorem \ref{thm:Neumann}, we can  also prove the following theorem whose proof is omitted.  
\begin{thm}
 Let $\lambda>0$ and  $(\alpha,p)\in \mathscr{A}\cap\mathscr{B}$.
\begin{enumerate}[label={\textnormal{(\roman*)}},ref={\roman*},leftmargin=2em]
\item Assume that $(\beta,q)$ satisfies
\[   \alpha\leq \beta,\quad  0<\frac{1}{q}<\beta-1 \leq 1   \quad\text{and}\quad\frac{1}{p}-\frac{\alpha}{n} \geq \frac{1}{q}-\frac{\beta}{n}. \]
 For every $f\in \Bessel{\beta-2}{q}(\Omega)$ and $u_{\Neu} \in \bes{\alpha-1-1/p}{p}(\partial\Omega)$, there exists a unique function $u\in \Bessel{\alpha}{p}(\Omega)$  such that
\begin{equation*}
\left\{\begin{alignedat}{2}
-\triangle u+\Div(u\boldb) +\lambda u& =f & \quad & \text{in }\Omega,\\
\gamma_{\nu} (\nabla u-u\boldb)& =u_{\Neu} & \quad & \text{on }\partial\Omega.\\
\end{alignedat}\right.
\end{equation*}
\item   Assume that $(\beta,q)$ satisfies
\[   \beta\leq \alpha,\quad 0\leq \beta<\frac{1}{q}<1,\quad \text{and }\quad  \frac{1}{p}-\frac{\alpha}{n} = \frac{1}{q}-\frac{\beta}{n}. \]
 For every $g\in \Bessel{-\beta}{q'}(\Omega)$ and $v_{\Neu} \in \bes{1/p-\alpha}{p'}(\partial\Omega)$, there exists a unique  function  $v\in \Bessel{2-\alpha}{p'}(\Omega)$  such that
\begin{equation*}
\left\{\begin{alignedat}{2}
-\triangle v-\boldb \cdot \nabla v+\lambda v& =g & \quad & \text{in }\Omega,\\
\gamma_{\nu} (\nabla v)& =v_{\Neu} & \quad & \text{on }\partial\Omega.
\end{alignedat}\right.
\end{equation*}
\end{enumerate}
\end{thm}

\providecommand{\bysame}{\leavevmode\hbox to3em{\hrulefill}\thinspace}
\providecommand{\MR}{\relax\ifhmode\unskip\space\fi MR }
\providecommand{\MRhref}[2]{%
  \href{http://www.ams.org/mathscinet-getitem?mr=#1}{#2}
}
\providecommand{\href}[2]{#2}

\end{document}